\documentclass[12pt, a4paper, fleqn]{amsart}
\makeatletter \@namedef{subjclassname@2020}{%
	
	\textup{2020} Mathematics Subject Classification} \makeatother

\usepackage[T1]{fontenc}
\usepackage[english]{babel}
\usepackage{fullpage}
\usepackage{amssymb}
\usepackage{mathptmx}
\usepackage[dvipsnames]{xcolor}
\usepackage{graphicx}

\newcommand{\C}{\mathbb{C}}
\newcommand{\N}{\mathbb{N}}

\newcommand{\R}{\mathbb{R}}
\newcommand{\aol}{\mathcal{A}}
\newcommand{\bol}{\mathcal{B}}
\newcommand{\pol}{\mathcal{P}}
\newcommand{\fol}{\mathcal{F}}

\newcommand{\de}{\mathrm{deg}\,}

\newcommand{\ca}{\mathrm{card}\,}

\newcommand{\be}{\begin{equation}}
\newcommand{\ee}{\end{equation}}

\newtheorem{proposition}{Proposition}[section]
\newtheorem{thm}[proposition]{Theorem}
\newtheorem{lemma}[proposition]{Lemma}
\newtheorem{corol}[proposition]{Corollary}
\theoremstyle{definition}
\newtheorem{definition}[proposition]{Definition}
\newtheorem{example}[proposition]{Example}
\newtheorem{problem}[proposition]{Problem}
\newtheorem{remark}[proposition]{Remark}

\usepackage[pdfstartview=FitH, pdfprintscaling=AppDefault, bookmarksopenlevel=4]{hyperref}

\begin{document}

\title[Polynomial meshes on algebraic sets]{Polynomial meshes on algebraic sets}
\author[L. Bia\l as-Cie\.{z}]{Leokadia Bialas-Ciez }
\address{Faculty of Mathematics and Computer Science, Jagiellonian University, {\L}ojasiewicza~6, 30-348 Krak\'ow, Poland}
\email{leokadia.bialas-ciez@uj.edu.pl}

\author[A. Kowalska]{Agnieszka Kowalska}
\address{Institute of Mathematics, Pedagogical  University of Krakow,
	Podchor\c{a}\.zych~2, 30-084 Krak\'ow, Poland}
\email{agnieszka.kowalska@up.krakow.pl}
\author[A. Sommariva]{Alvise Sommariva}
\address{Department of Mathematics {{Tullio Levi-Civita}}, University of Padua, Via Trieste 63, 35121 Padua, Italy}
\email{alvise@math.unipd.it}

\thanks{The work was partially supported by the National Science Center, Poland grant No. 2017/25/B/ST1/00906}

\begin{abstract} Polynomial meshes (called sometimes 'norming sets') allow us to estimate the supremum norm of polynomials on a fixed compact set by the norm on its discrete subset. We give a general construction of polynomial weakly admissible meshes on compact subsets of arbitrary algebraic hypersurfaces in $\C^{N+1}$. They are preimages by a projection of meshes on compacts in $\C^N$. The meshes constructed in this way are optimal in some cases. Our method can be useful also for certain algebraic sets of codimension greater than one. To illustrate applications of the obtained theorems, we first give a few  examples and finally report some numerical results. In particular, we present  numerical tests (implemented in Matlab), concerning the use of such optimal polynomial meshes for interpolation and least-squares approximation, as well as for the
evaluation of the corresponding Lebesgue constants.
\end{abstract}
\maketitle
\section{Introduction}

The problem with proper estimation of a norm of polynomials on a given compact set by the norm on a discrete subset is an exciting subject of research because of a large number of applications. First results of this type were shown more than 100 years ago but some new insights were given in 2008 by Calvi and Levenberg. Since then, "{\it admissible meshes}" defined in \cite{CL} have become a subject of numerous  investigations, see e.g., \cite{BDMSV}, \cite{Pi16}, \cite{Bo18}, \cite{Vi18}, \cite{Kr19} and the references therein. They are sometimes called {\it norming sets, polynomial meshes} or {\it Marcinkiewicz-Zygmund array}. Due to their nice properties  (see e.g., \cite{LSV}), they play a relevant role in recent multivariate interpolation and approximation (see e.g., \cite{CL}, \cite{BDMSV}).

\begin{definition}
	Let $K$ be a compact set in $\C^N$.
	We say that a family of finite subsets $\{\aol_n\}_n$ of $K$ represent {\it norming sets} on $K$ if for all $p\in\pol_n(\C^N)$ (polynomials of degree at most $n$ of variable $z\in\C^N$) we have the inequality
	\be \label{mesh} \|p\|_K \ \le \ C_n \: \|p\|_{\aol_n}
	\ee
	where $\|\cdot\|_K$ and $\|\cdot\|_{\aol_n}$ are uniform norms on $K$ and $\aol_n$ respectively, and $C_n$ is a positive constant depending only on $n$ and $K$. 	Moreover,
	\begin{itemize}
		\item if $C_n$ is independent of $n$ and $\ca\aol_n$ are of polynomial growth in $n$ then  $\{\aol_n\}_n$ are called {\it admissible meshes} (AM for short),
		\item if $C_n$ and $\ca\aol_n$ are of polynomial growth in $n$ then  $\{\aol_n\}_n$ are called {\it weakly admissible meshes} (WAM for short),
		\item if $C_n$ is independent of $n$ and $\ca\aol_n=\mathcal{O}({\rm dim}\,\pol_n(K))$ then  $\{\aol_n\}_n$ are called {\it optimal polynomial meshes} (OPM for short). The space $\pol_n(K)$ can be defined as $\pol_n(\C^N)_{|_K}$.
	\end{itemize}
\end{definition}

Observe that, by inequality (\ref{mesh}), any set $\aol_n$ is {\it determining for polynomials} $\pol_n(K)$, i.e., the condition $p \equiv 0$ on $\aol_n $ implies $p\equiv 0$ on $K$ for any $p\in \pol_n(K)$. Therefore, ${\rm dim}\,\pol_n(K)$ is the minimal cardinality of $\aol_n$. A set $K$ is called {\it polynomially determining} if it is determining for the whole space $\pol(\C^N)=\bigcup_{n=0}^\infty \pol_n(\C^N)$.

\pagebreak

Optimal polynomial meshes have been constructed on many polynomially determining compact sets, e.g., sections of discs, ball, convex bodies, sets with regular boundary (see \cite{Vi18}, \cite{Pi16}, \cite{Kr19} and the references given therein). Regarding subsets of algebraic varieties (that, of course, are not polynomially determining),  norming sets are known only for a few cases. 

\vskip 5mm

\begin{problem} How to construct polynomial meshes on compact subsets of arbitrary algebraic sets in $\C^N$?
\end{problem}


The problem seems to be intricate and difficult. So far it has only been solved for a few 
compacts like sections of a sphere, a torus, a circle and  curves in $\C$ with analytic parametrization (see, e.g., \cite{SV21}, \cite{SV18},  \cite{Vi18'}, \cite{PV10}). In these cases, admissible meshes are transferred by some analytical map from certain polynomially determining set with sufficiently relevant meshes.     

\vskip 1mm

In this paper we give a general construction of polynomial meshes on arbitrary algebraic hypersurface in $\C^N$. The method can also be applied to some algebraic sets of codimension greater than one.  

\vskip 1mm

More precisely, we construct weakly admissible meshes on a compact subset $E$ of an algebraic variety $V$ by means of meshes $\{\aol_\ell\}_\ell$ on a projection of $E$ (see Theorems \ref{ThmAdmMesh} and \ref{ThmAdmMesh2}). We show that preimages of $\{\aol_\ell\}_\ell$ by the projection represent norming meshes on $E\subset V$. For some specific algebraic sets we obtain optimal polynomial meshes (Theorem \ref{yk=s0} and Corollary \ref{CorOptAdm1}). In the case of algebraic hypersurfaces the problem seems to be solved  but for algebraic sets of codimension greater than one, it is much more complicated. In this case, we give some partial results (see Theorem \ref{TwiceSpecific}, Corollary \ref{CorOptAdmCodim} and Theorem \ref{ThmAdmMesh2}) and a few  examples to show how to use our theorems in more complicated cases.           

\vskip 1mm

Consider an arbitrary algebraic hypersurface $V\subset \C^{N+1}$ defined as the zero level set of a~polynomial $s=s(z,y)$ where $z=(z_1,...,z_N)$, $y\in \C$. Without loss of generality (taking a linear change of variables if necessary) we can assume that 
\be \label{syk}
s(z,y)=y^k+\sum_{j=0}^{k-1} \varsigma_j(z)\, y^j,
\ee 
for some $k\in \N=\{1,2,...\}$ and polynomials $\varsigma_0,...,\varsigma_{k-1}\in \pol(\C^N)$.  In this sense, any algebraic set of codimension one is given by a polynomial $s$ in the above form. 

\vskip 1mm

We denote by $\pol(\C^N)$ the space of all polynomials of variable $z\in \C^N$. We will sometimes write $\pol(z)$ for this space and $\pol_n(z)$ for $\pol_n(\C^N)$. To give a precise definition of $\pol_n(K)$ we take the ideal $I(K)$ of polynomials from $\pol_n(\C^N)$ equal to zero on the set~$K\subset\C^N$ and set $\pol_n(K)$ as the quotient space $\pol_n(\C^N)/I(K)$. We define $\pol(K)$ in a similar way.


\vskip 1mm

For an arbitrary algebraic set $V\subset \C^N$ and $p\in \pol(V)$ let
\be \label{def_deg_V} \de_V p:=\min\{\de q \: : \: q\in \pol(\C^N), \ q_{|_V}\equiv p\}.
\ee
In Section 2  we will prove estimates of $\de p$ by $\de_V p$ for any algebraic hypersurface $V$. We will also give such estimates for some algebraic sets of codimension greater than one. These estimates will play a crucial role in the next section devoted to constructions of meshes on compact subsets of algebraic varieties. Applications of these theoretical results will be presented in the last section where we give some examples of optimal polynomial meshes and weakly admissible meshes on compact subsets of algebraic sets.    

\vskip 3mm

\section{Estimation of degree of polynomials on algebraic sets}

\vskip 3mm 

\subsection{The basic lemma} Consider a polynomial $s$ of the form (\ref{syk}) of total degree  $d:=\de s \ge k$. Let 
\[ p(z,y)=\sum_{j=0}^{k-1} p_j(z)\, y^j \]
for $z\in \C^N$, $y\in \C$ and some given polynomials $p_0,...,p_{k-1}\in \pol(z)$. Fix an arbitrary polynomial $w\in\pol(z,y)$  and set 
\be \label{q=p+sw} q=p+sw. \ee 
We will estimate \ $\de p$ \ by \ $\de q$ \ in the following way.   

\begin{lemma} \label{lemat ze stopniami}
	With the above notations, if $w\not\equiv0$ then the inequality
		\[ \de p \le d\cdot\de q -dk+d \]
holds and is optimal. Moreover,
	\[ {\it if} \  \ d=k \ \ {\it then} \ \ \de p\le \de q.\]
\end{lemma}

The first estimate is optimal, because for a polynomial $s$ such that $d>k$ and \ $p(z,y)=s(z,y)-y^k$, \ $w(z,y)=-1$ \ we have \ $q(z,y)=-y^k$ and 
	\[ \de p=d =d\cdot \de q -dk +d.\] 

\begin{proof}
	Observe that if the sharp inequality $\de q<k$ were true then $w\equiv0$ contrary to assumption. Therefore, we take up only the case $k\le \de q$. 
	Let $n$ be the degree of $w$ with respect to $y$ and 
	\[ w(z,y)=\sum_{j=0}^n w_j(z)\,y^j \]
	for some polynomials $w_0,...,w_n\in \pol(z)$.	Observe that 
\begin{align*}
q(z,y)=&\sum_{j=0}^{k-1} p_j(z)\, y^j + \left(y^k+\sum_{j=0}^{k-1} \varsigma_j(z)\, y^j\right) \sum_{j=0}^n w_j(z)\,y^j\\
=&\sum_{j=0}^{k-1} \left(p_j(z) + \sum_{i=0}^{j} \varsigma_{j-i}(z) w_i(z) \right)  y^j  + \sum_{j=0}^{n-1} \left(w_j(z) + \sum_{\ell=j+1}^{n} \varsigma_{k+j-\ell}(z) w_\ell(z) \right) y^{j+k}\\
&+ w_n(z)\,y^{k+n}
\end{align*}
where we set $\varsigma_m=0$ for $m<0$.
By assumption, $w\not\equiv 0$ and consequently,
\begin{align*}
\de q=&\max\left\{k+n+\de w_n ,\ \max_{j=0,...,n-1} \left[ j+k + \de \left( w_j + \sum_{\ell=j+1}^{n} \varsigma_{k+j-\ell} w_\ell \right)\right],\ \ \ \right.\\
	 &\ \ \ \ \ \left.\max_{j=0,...,k-1} \left[ j+\de \left( p_j + \sum_{i=0}^{j} \varsigma_{j-i} w_i \right) \right]\right\}.
\end{align*} 
	This yields
	\be \label{k+n} \de q \ge k+n+\de w_n.
	\ee
	If $\de p\le \de q$, we get
	\[ \de p \le \de q +(d-1)k-dk+k\le d\cdot \de q-dk+k\le d\cdot \de q-dk+d. \]
	Therefore, it is sufficient to consider only the case $\de q<\de p$. 
	
	Since 
	\[ \de p=\max_{j=0,...,k-1} (j+\de p_j), \]
	we take such $\ell\in \{0,...,k-1\}$ that $\de p=\ell +\de p_\ell$. It follows that 
	\[ \ell +\de \left( p_\ell + \sum_{i=0}^{\ell} \varsigma_{\ell-i} w_i \right) \le \de q < \de p = \ell +\de p_\ell 	\]
	and so
	\[ \de \left( p_\ell + \sum_{i=0}^{\ell} \varsigma_{\ell-i} w_i \right) < \de p_\ell
	\]
	that implies
	\[ \de p_\ell = \de \sum_{i=0}^{\ell} \varsigma_{\ell-i} w_i \le \max_{i=0,...,\ell}  \de(\varsigma_{\ell-i} w_i)\le \max_{i=0,...,\ell} (d-\ell+i+\de w_i). 	\] 
	Hence
	\be \label{d+i} \de p \le d+i+\de w_i \ \ \ \ \ \mbox{for some} \ \ i\in\{0,...,k-1\}.\ee
	If \ $\de p\le \de q +nd+d-k-n$, \ then from (\ref{k+n}) we have
\begin{align*}
\de p&\le d\cdot \de q-(d-1)\de q +nd+d-k-n \\
&\le d\cdot \de q-(d-1)(k+n) +nd+d-k-n = d\cdot \de q-dk+d
\end{align*}
	and the assertion follows. 	Therefore, we will consider only the case  \[\de p > \de q +n(d-1)+d-k\]
	that yields
	\be \label{de_p>} \de p > \de q +j(d-1)+d-k \ \ \ \ \ \mbox{for all} \ \ j\in\{0,...,n\}.
	\ee
	Let
	\[ A:= \{ \ell\in\{0,...n \}\: : \: \de p\le (\ell+1)d+\de w_\ell\}. \]
	By inequality (\ref{d+i}), we see that the set $A$ is not empty and we can take \ $a:=\max A$. \ We now prove that $a=n$. Indeed, suppose otherwise. Then $a<n$ and inequality (\ref{de_p>}) would imply
	\[ \de w_a\ge \de p -ad-d > \de q +ad-a +d-k-ad -d = \de q -a-k. \]
	Since \ $\de q\ge a+k + \de \left(w_a+\sum\limits_{\ell=a+1}^{n} \varsigma_{k+a-\ell} w_\ell\right) $, \ 
	we would have 
	\[ \de w_a > \de \left(w_a+\sum\limits_{\ell=a+1}^{n} \varsigma_{k+a-\ell} w_\ell\right). \]
	Thus
	\be \label{deg_wa} \de w_a = \de \left(\sum\limits_{\ell=a+1}^{n} \varsigma_{k+a-\ell} w_\ell\right) \le \max_{\ell=a+1,..,n} (d-k-a+\ell +\de w_\ell) \ee
	and so, for some $\ell\in\{a+1,...,n\}$,
\begin{align*}
\de w_\ell \ge&\ \de w_a -d +k +a-\ell \ge \de p -(a+1)d-d+k+a-\ell\\ =&\ \de p -(\ell+1)d + (d-1)(\ell-a)-d+k \ge \de p -(\ell+1)d -1+k\\ \ge&\ \de p -(\ell+1)d
\end{align*}
	which is impossible because $\ell\not\in A$.
	
By the fact that $a=n$ and from inequality (\ref{k+n}) we get
\begin{align*}
\de p \le &(n+1)d +\de w_n \le nd+d +\de q -k-n =d\cdot \de q-(d-1)\de q +nd+\\&+d-k-n\le d\cdot \de q -(d-1)(k+n) +nd +d-k-n	
\end{align*}  
	and the assertion follows. 
	
	For $d=k$ we obtain a better estimate of $\de p$ than this one given in the first case. 
	Let 
	\[ A:=\{\ell\in \{0,...,n\} \: : \: \de p \le \de w_\ell + \ell +k \}. \]
	If the inequality $\de p\le \de q$ did not hold, then (\ref{d+i}) would imply $A\ne\emptyset$ and we could consider $a:=\max A$. From (\ref{k+n}) we would obtain $n\not\in A$. By inequality (\ref{deg_wa}), for some $\ell\in\{a+1,...,n\}$ we would have   
	\[ \de w_\ell \ge \de w_a -d +k +a-\ell \ge \de p -a-k -d+k+a-\ell = \de p  -k-\ell  \] and $\ell$ would belong to $A$ which is a contradiction. 	
\end{proof}

\begin{proposition} \label{Prop_stopnie} With the previous notations we have 
	\[ \de p \le  \max \{ \de q, \ d\cdot \de q-dk+d\} \le d\cdot\de q.
	\] 
	Moreover, if \ $d=k$ \ then \ $\de p \le \de q$. 
\end{proposition}

\vskip 2mm

\subsection{On algebraic hypersurfaces} 

Let $V=V(s)$ be an algebraic set of codimension one, defined by a polynomial $s$, i.e.
\[ V=\{(z,y)\in \C^N\times \C \: : \:s(z,y)=0\}.\] 
Without loss of generality we can assume that $s$ is of form (\ref{syk}).
Since $\pol(V)$ is a quotient space, we can show that any equivalence class from $\pol(V)$ is uniquely determined by a polynomial of the space  
\be \label{def_W} W=\pol(z) \otimes \pol_{k-1}(y):=\{p(z,y)=\sum_{j=0}^{k-1}p_j(z) \: y^j \: : \: p_0,...,p_{k-1}\in\pol(z)\}, \ee
cf. Subsection 4.2 in \cite{BCCK1}. 
Dividing a fixed polynomial $q\in \pol(z,y)$ by $s$ with respect to the variable $y\in \C$, we can find $p\in \pol(z) \otimes \pol_{k-1}(y)$ and $w\in\pol(z,y)$ satisfying condition ({\ref{q=p+sw}) and assumptions of Proposition \ref{Prop_stopnie}. 
The polynomial $p$ is unique and in this way we obtain the linear surjective operator 
\[ \Phi: \pol(z,y) \ni q \mapsto p\in \pol(z) \otimes \pol_{k-1}(y)\] 
such that $q_{|_V}=\Phi(q)_{|_V}$. By the proposition, we can estimate $\de \Phi(q)$ 
\[ \de_V q \le \de \Phi(q) \le \max \{ \de_V q, \ d\cdot \de_V q-dk+d\} \le d\cdot\de_V q
\]
where $\de_V q$ is defined by (\ref{def_deg_V}) and $d=\de s$. We can state equivalently the above inequalities as follows

\begin{thm}
	For all polynomials $p\in \pol(z) \otimes \pol_{k-1}(y)$ we have
	\[ \de_V p \le \de p \le  \max \{ \de_V p, \ d\cdot \de_V p-dk+d\} \le d\cdot\de_V p.
	\] 
	Moreover, if \ $d=k$ \ then \ $\de_V p = \de p$. 
\end{thm}

In the specific case where $s(z,y)=y^k+\varsigma_0(z)$, a better estimate of $\de p$ can be found in \cite{BCCK1}.

\begin{thm} \label{deg ze starej pracy} (\cite[Lemma 12]{BCCK1}) \ 
	If \ $V=\{(z,y)\in \C^N\times \C \: : \: y^k+\varsigma_0(z)=0\}$ \ then for any polynomial 
	$p\in \pol(z) \otimes \pol_{k-1}(y)$ we have
	\[ \de_V p \le \de p \le  \max \left\{ 1, \frac{\de \varsigma_0}{k}\right\}\de_V p \ = \ \frac{d}k \ \de_V p
	\] 
	where $d=\de (y^k+\varsigma_0(z))$.
\end{thm}

\begin{remark} \label{dim for hyperalg}
For numerical tests given in Section 5, we need to know what the dimension of the restriction of $\mathcal{P}_n(\mathbb{C}^{N+1})$ to $V$ is. Consider an algebraic hypersurface $V\subset\mathbb{C}^{N+1}$ given by a polynomial equation $s(z,y)=0$, where $z=(z_1,...,z_N)\in \mathbb{C}^N$ and the total degree of $s$ equals $k$. Without loss of generality (taking a linear change of variables if necessary), we can assume that 
\[ s(z,y) = y^k + \sum_{j=0}^{k-1} \varsigma_j(z)\: y^j  \]
and $\varsigma_j\in\mathcal{P}_{k-j}(z)=\mathcal{P}_{k-j}(\mathbb{C}^N)$ for $j=0,...,k-1$. In this case, 
$\mathcal{P}(V)_{|_V}=W_{|_V}$ with $W$ defined by (\ref{def_W}). 
Consider $W_n:=\{p\in W\: : \: {\rm deg}\,p\le n\}$ where deg$\,p$ is the total degree of $p$. The set $W_n$ is a vector space and we can calculate its dimension as follows. If $n\le k-1$ then $W_n=\mathcal{P}_n(\mathbb{C}^{N+1})$ and dim$\,W_n=\binom{N+1+n}{N+1}$. For $n\ge k$, since deg$\,p\le n$ implies deg$\,p_j\le n-j$ for $j=0,...,k-1$, we have the following formula
\[ {\rm dim} \, W_n = \sum_{j=0}^{k-1} {\rm dim}\, \mathcal{P}_{n-j}(\mathbb{C}^{N}) = \sum_{j=0}^{k-1} \binom{N+n-j}{n-j} = \binom{N+1+n}{N+1} - \binom{N+1+n-k}{N+1}. \]
\end{remark}

\vskip 2mm 

\subsection{On some algebraic sets of codimension greater than one} Now consider two polynomials $s_1\in\pol(z,y_1)$, $s_2\in\pol(z,y_1,y_2)$ for $z\in \C^N$, $y_1,y_2\in\C$ such that 
\be \label{s1yk}
s_1(z,y_1)=y_1^{k_1}+\sum_{j=0}^{k_1-1} \tilde{\varsigma}_j(z)\, y_1^j, \ \ \ \ \ d_1:=\de s_1 \ge k_1
\ee \vskip -4mm
\be \label{s2yk}
s_2(z,y_1,y_2)=y_2^{k_2}+\sum_{j=0}^{k_2-1} \tilde{\tilde{\varsigma}}_j(z,y_1)\, y_2^j, \ \ \ \ \ d_2:=\de s_2 \ge k_2,
\ee 
where $\tilde{\varsigma}_0,...,\tilde{\varsigma}_{k_1-1}\in \pol(z)$, $\tilde{\tilde{\varsigma}}_0,...,\tilde{\tilde{\varsigma}}_{k_2-1}\in \pol(z,y_1)$ and  $k_1,k_2\in \N$ are fixed.
The polynomials $s_1,s_2$ determine the algebraic set $V=V(s_1,s_2)=\{s_1=0,s_2=0\}$ of codimension two in $\C^{N+2}$. 
Take an arbitrary polynomial $q\in \pol(z,y_1,y_2)$ and divide it by $s_2$ with respect to $y_2$. In this way we find polynomials 
\[ r(z,y_1,y_2)=\sum_{j=0}^{k_2-1} r_j(z,y_1)\, y_2^j  \ \ \ \ \mbox{and} \ \ \  w\in\pol(z,y_1,y_2) \ \ \mbox{such that} \ \ q=r+s_2\cdot w.\]  
By Proposition \ref{Prop_stopnie}, we have
\[ \de r \le \max\{ \de q, \: d_2 \: (\de q - k_2+ 1)\} \le d_2 \: \de q \]  
and if \ $d_2=k_2$ \ then \ $\de r\le \de q$. \ 
Now we divide every polynomial $r_j$ by $s_1$ to obtain $p_j$ and $w_j$ such that
\[ r_j = p_j + s_1 \cdot w_j \ \ \ \ \ \mbox{and } \ \ \ \ \ \de _{y_1}p_j \le k_1-1.\] 
Let 
\[ p(z,y_1,y_2)=\sum_{j=0}^{k_2-1} p_j(z,y_1)\, y_2^j \ \ \in \pol(z)\otimes\pol_{k_1-1}(y_1)\otimes\pol_{k_2-1}(y_2).\]
Again by Proposition \ref{Prop_stopnie}, we get the estimate
\begin{align*}
\de p&=\max_{j=0,...,k_2-1}(j+\de p_j) \le \max_{j=0,...,k_2-1} \{ j+ \max\{\de r_j,\: d_1(\de r_j-k_1+1 )\}\} \\ &\le \max\{ \de r,\: d_1 \,( \de r -k_1+1)\}  \\ &\le \max\{ \de q, \: d_1 ( \de q -k_1+1), \: d_1 \,[ d_2 \: (\de q - k_2+ 1) -k_1+1]\}
\end{align*} 
that holds in a general case. For some specific cases depending on $d_1,k_1, d_2,k_2$ we have
\[ d_1=k_1 \ \mbox{and} \ d_2=k_2 \ \Rightarrow \ \de p\le \de q\]
\[ d_i>k_i \ \mbox{and} \ d_j=k_j \ \Rightarrow \ \de p\le \max \{\de q, \: d_i\:(\de q -k_i+1)\} \] 
for $i=1$ or $i=2$, $j\ne i$.

\begin{thm}
	For all polynomials $p\in \pol(z)\otimes\pol_{k_1-1}(y_1)\otimes\pol_{k_2-1}(y_2)$ 
	\begin{align*}
\de_V p &\le \de p 
	 \le \max\{ \de q, d_1 ( \de q -k_1+1), d_1 [ d_2 (\de q - k_2+ 1) -k_1+1]\}\\ &\le d_1\, d_2\, \de_V p.	
\end{align*} 
	Moreover, \[ \mbox{if} \ \ d_1=k_1 \ \ \mbox{then} \  \ \de_V p \le \de p \le  d_2\: \de_V p,\]
	\[ \mbox{if} \ \ d_2=k_2 \ \ \mbox{then} \  \ \de_V p \le \de p \le  d_1\: \de_V p,\]
	\[ \mbox{if} \ \ d_1=k_1 \ \ \mbox{and} \ \ d_2=k_2 \ \ \mbox{then} \  \ \de_V p = \de p .\]
\end{thm}

\vskip 2mm 

Obviously, one can easily state an analogous result for algebraic sets of codimension $\ell>2$ in $\C^{N+\ell}$ given by polynomials $s_1\in\pol(z,y_1),...,s_\ell\in\pol(z,y_1,...,y_\ell)$.

\begin{remark} \label{dim for alg}
\noindent Similarly to the end of the previous subsection, we give formulas for the dimension of a space of polynomials restricted to an algebraic surface $V$ of codimension 2. Let $V\subset\mathbb{C}^{N+2}$ be an algebraic set defined by two polynomial equation $s_1(z,y)=0$, $s_2(z,y,x)=0$  where $z=(z_1,...,z_N)\in \mathbb{C}^N$ and $y,x\in \mathbb{C}$. Let $k$ be the total degree of $s_1$ and $\overline{k}$ of $s_2$. Assume that  
\[ s_1(z,y)=y^{k}+\sum_{j=0}^{k-1} \tilde{\varsigma}_j(z)\, y^j \ \ \ \ \ \ \mbox{and} \ \ \ \ \ \  
s_2(z,y,x)=x^{\overline{k}}+\sum_{j=0}^{\overline{k}-1} \tilde{\tilde{\varsigma}}_j(z,y)\, x^j . \]
In this case, 
$\mathcal{P}(V)_{|_V}=\overline{W}_{|_V}$ with $\overline{W}$ defined by
\[ W = \left\{ p(z,y,x) = \sum_{j=0}^{\overline{k}-1} p_j(z,y) \: x^j \: : \: p_0, ..., p_{\overline{k}-1} \in W \right\}, \]
where $W$ is given in (\ref{def_W}).
Consider $\overline{W}_n:=\{p\in W\: : \: {\rm deg}\,p\le n\}$ where deg$\,p$ is the total degree of $p$. The dimension of the space $\overline{W}_n$ can be calculated as follows. In the case  $n\le \overline{k}-1$ we have 
\[ \mbox{dim}\,\overline{W}_n= \mbox{dim}\,W_n + \mbox{dim}\,W_{n-1} +...+ \mbox{dim}\,W_0,
\]
and for $n > \overline{k}-1$ we get
\[ \mbox{dim}\,\overline{W}_n= \mbox{dim}\,W_n + \mbox{dim}\,W_{n-1} +...+ \mbox{dim}\,W_{n-\overline{k}+1}. \]
From this, by Remark \ref{dim for hyperalg}, we can obtain precise formulas. For example, if $n\le \min\{k,\overline{k}\}-1$ then $\mbox{dim}\,\overline{W}_n = \binom{N+2+n}{N+2}$. The most useful case seems to be $n\ge k+\overline{k}-1$ when we get
\[ \mbox{dim}\,\overline{W}_n=  \binom{N+1+n}{N+1} - \binom{N+1+n-k}{N+1} + \binom{N+n}{N+1} - \binom{N+n-k}{N+1} \] \[ + \: ... \: + \binom{N+2+n-\overline{k}}{N+1} - \binom{N+2+n-k-\overline{k}}{N+1} \] \[ = \binom{N+2+n}{N+2} - 2\binom{N+2+n-k}{N+2}+ \binom{N+2+n-k-\overline{k}}{N+2} .\]   
\end{remark}

\vskip 5mm

\section{Construction of polynomial meshes on algebraic sets}

\vskip 3mm

\subsection{Division inequality} 
We will prove that preimages by a projection of some norming sets  $\{\aol_\ell\}_\ell\subset K\subset \C^N$ 
represent weakly admissible meshes on a compact subset $E$ of an algebraic variety $V$. These meshes are optimal for some specific algebraic sets. In this case we assume only the existence of admissible meshes on $K$ but for more general varieties we need also a~property of $K$ that is called a {\it division} or {\it Schur-type inequality}. 

 \vskip 2mm
\begin{definition}
Let $K$ be a compact subset of $\mathbb{C}^N$ and $m>0$. We say that $K$ satisfies the {\it division inequality} with exponent $m$ if for any polynomial $q\not\equiv 0$ on $K$ there exists a positive constant $M$ such that for all polynomials $p\in \mathcal{P}(\mathbb{C}^N)$
\begin{equation}\label{division}
\|p\|_K \le M(\deg p + \deg q)^{m\: \deg q}\|pq\|_K
\end{equation}
where $\|\cdot\|_K$ is the sup-norm on the set $K$. 
\end{definition}

One of tools often used to construct norming sets is a so-called Markov property \linebreak 
(i.e., \ $\|{\rm grad}\, p\|_K\le M(\de p)^m\|p\|_K$ with $M,m>0$ independent of $p\in\pol(\C^N)$), see e.g., \cite{CL}. 
In \cite{BCCK1} we proved that any set $K\subset\C^N$ with Markov property satisfies the division inequality, cf. \cite{BC99}. The converse is not possible for several variables, because every Markov set is polynomially determining and this is not a necessary condition for the division inequality. A list of sets satisfying the division inequality contains 
\begin{itemize}
	\item all non-singular connected compacts in the complex plane,
	\item all convex fat compact subsets of $\R^N$,
	\item all compact sets with $\mathcal{C}^2$ boundary, 
	\item all fat subanalytic set in $\R^N$,
	\item all uniformly perfect sets,
	\item filled in Julia sets and some totally disconnected sets, e.g., Cantors,
	\item variate examples of compact subsets of algebraic varieties, see \cite{BCCK1}, \cite{BCCK2}.
\end{itemize}
These last mentioned sets satisfies the division inequality but they do not have Markov property, because algebraic sets are not polynomially determining. 

\vskip 3mm

\subsection{On algebraic hypersurfaces} 
Assume that $s$ is a polynomial in form (\ref{syk}) and set 
\begin{equation}\label{A}
\mathcal{A}=\mathcal{A}(s):=\{z\in\mathbb{C}^N\ :\ {\rm Disc}_y\left(s\right)=0\}=\left\{z\in\mathbb{C}^N\ :\ {\rm Res}_y\left(s,\frac{\partial s}{\partial y}\right)=0\right\}
\end{equation}
where ${\rm Disc}_y\left(s\right)$ is the discriminant of $s$ with respect to $y$ and ${\rm Res}_y\left(s,\frac{\partial s}{\partial y}\right)$ is the resultant of $s,\frac{\partial s}{\partial y}$ in $y$. 
Since the resultant is equal (with accuracy to $\pm$) to the determinant of the Sylvester matrix associated to the polynomials $s,\:\frac{\partial s}{\partial y}$ (see, e.g., \cite[Chap.1]{DE}), the set $\mathcal{A}(s)$ is an algebraic hypersurface in $\mathbb{C}^N$ and we can easily check whether a fixed point $z_0$ belongs to it or not, what is important in Theorem \ref{ThmAdmMesh}.

The following remark reveals the importance of the set $\aol(s)$: the polynomial \linebreak $\C\ni y\mapsto s(z,y)$ has $k$ roots pairwise distinct if and only if $z\in\mathbb{C}^{N}\setminus\aol(s)$. Irreducibility of $s$ implies $\mathbb{C}^{N}\setminus\aol(s)\neq\emptyset$. However, this is not a necessary condition, because the polynomial $w(z,y)=y^2-z^2$, \ $z,y\in \mathbb{C}$ has two distinct roots $y_1,y_2$ for any $z\in \mathbb{C}\setminus \{0\}$. For more details we refer the reader to \cite{BCCK2}. 

Consider a compact set $K\subset \C^N$ and an algebraic hypersurface $V=V(s)\subset \C^{N+1}$ given by a~polynomial $s$ in form (\ref{syk}) of degree $d$. Observe that the set
\be \label{E jako podniesienie K} E=\{(z,y)\in V(s) \:  : \:  z\in K\} 
\ee
is compact. We will transfer norming sets from $K$ to $E$. A construction  of weakly admissible meshes on $E$ is based on the following lemma which has been obtained by a careful analysis of the proof of Proposition 3.3  in \cite{BCCK2} 

\begin{lemma}\label{Lemat za Proposition}
Let $K$ be a compact subset of $\C^N$. There exists a polynomial $\widetilde{q}\in\pol(z)$ such that $\{z\in\mathbb{C}^N\ :\ \widetilde{q}(z)=0\}\subset\mathcal{A}(s)$, $\deg \widetilde{q} \le 2d(k-1)$ and if $F$ is compact subset of $K$ such that $F\setminus\mathcal{A}(s)\neq\emptyset$ then
\[
\|p_j^2\widetilde{q}\|_F\leq \widetilde{C}^2 \|p\|_G^2
\]
for any polynomial $p\in \pol(z) \otimes \pol_{k-1}(y)$ written in the form $p(z,y)=\sum\limits_{j=0}^{k-1} p_j(z)\, y^j$ where $G=\{(z,y)\in V(s) \:  : \:  z\in F\}$ and  a constant $\widetilde{C}$ is independent of $p$ and $F$.
\end{lemma}

\begin{thm}\label{ThmAdmMesh} Let $K\subset \mathbb{C}^N$ be a compact with weakly admissible meshes $\{\mathcal{A}_n\}_{n}$ and $E\subset V(s)$ be a set given by (\ref{E jako podniesienie K}). Assume that $K\setminus\aol(s)\ne\emptyset$ and take $z_0$ from this set. If $K$ satisfies the division inequality (\ref{division}) then $\{\mathcal{B}_n\}_{n}$ defined by
\[\mathcal{B}_n=\{(z,y)\in E\ :\ z\in \mathcal{A}_{\ell}\cup \{z_0\}\}
\ \ \mbox{ where } \ \ \ell =\ell(n)=\left\{\begin{array}{ll} 2(n + k^2-k)& \mbox{ ; } \ d=k,\\
2dn& \mbox{ ; } \ d>k
\end{array}\right.\]
represent weakly admissible meshes on $E$, i.e., for every polynomial $p\in\pol_n(z,y)$ we have
\[\|p\|_E \ \le \ D_n \: \|p\|_{\mathcal{B}_n}\] 
and $\{D_n\}_n$ are of polynomial growth in $n$.
\end{thm}

\begin{proof}
For a fixed polynomial $p\in\pol_n(\mathbb{C}^{N+1})$ we can take $q\in\pol(z) \otimes \pol_{k-1}(y)$ such that $q=p$ on $V(s)$ and by Lemma \ref{lemat ze stopniami}, $\de q \le n$ if $d=k$ and $\de q \le d (n -k+1)$ if $d>k$.
Set $q(z,y)=\sum_{j=0}^{k-1} q_j(z) \,y^j$ for some polynomials $q_j\in \pol(z)$. 
We have
\[\|p\|_E=\|q\|_E\le \sum_{j=0}^{k-1}\|q_j\|_K\|y\|_E^j.\]
From Lemma \ref{Lemat za Proposition} there exists  a polynomial $\widetilde{q}$ such that $\{z\in\mathbb{C}^N\ :\ \widetilde{q}(z)=0\}\subset\mathcal{A}(s)$, $\deg \widetilde{q} \le 2d(k-1)$. Since $K\setminus\mathcal{A}(s)\neq\emptyset$, we have $\widetilde{q}\not\equiv 0$ on $K$. By the division inequality, for $j=0,\ldots,k-1$ we have 
\begin{equation}\label{qjonK}\|q_j^2\|_K\le M(2\deg q + 2d(k-1))^{2md(k-1)}\|q_j^2\widetilde{q}\|_K.
\end{equation}
Since $\{\mathcal{A}_n\}_{n}$ are weakly admissible meshes on the set $K$, we obtain
\be\label{qj2WAM} \|q_j^2\widetilde{q}\|_K\le C_{\ell}\|q_j^2\widetilde{q}\|_{\mathcal{A}_\ell\cup\{z_0\}} \ \mbox{ where } \ \ \ell =\ell(n)=\left\{\begin{array}{ll} 2(n + k^2-k)& \mbox{ ; } \ d=k,\\
2dn& \mbox{ ; } \ d>k
\end{array}\right.\ee
and $C_\ell=C_{\ell(n)}$ are of polynomial growth in $\ell$ and hence so are in $n$. By Lemma \ref{Lemat za Proposition},
\[\|q_j^2\widetilde{q}\|_{\mathcal{A}_\ell\cup\{z_0\}}\le \widetilde{C}^2 \|q\|_{\mathcal{B}_n}^2. \]
On combining this with (\ref{qjonK}) and (\ref{qj2WAM}), we deduce that
\[\|q_j\|_K\le (MC_\ell)^{1/2} \: \widetilde{C}\ \:(\ell(n))^{md(k-1)}\|q\|_{\mathcal{B}_n}.\]
Consequently,
\[\|p\|_E=\|q\|_E\ \le  D_n \: \|p\|_{\mathcal{B}_n}\]
and $D_n=(MC_\ell)^{1/2} \: \widetilde{C} \: (\ell(n))^{md(k-1)}\sum_{j=0}^{k-1}\|y\|_E^j$ are of polynomial growth in $n$.
\end{proof}

\vskip 2mm

\subsection{A specific case}  In the case of $\varsigma_1=...=\varsigma_{k-1}=0$, the division inequality on $K$ is not necessary to construct meshes on $E$. We just take admissible meshes on $K$ and lift them on the algebraic varieties $V(s)$. We will see that in this fashion we can obtain optimal polynomial meshes. It is worth recalling that for an algebraic hypersurface in $\C^{N}$ we have \ dim$\, \pol_n(V)=\mathcal{O}(n^{N-1})$ \ and for an algebraic set $V$ of codimension $m$, we have \ dim$\, \pol_n(V)=\mathcal{O}(n^{N-m})$. 

\begin{thm} \label{yk=s0}
	Let $V=V(s)\subset \C^{N+1}$ be an algebraic variety defined by a polynomial of the form \ $s(z,y)=y^k-\varsigma_0(z)$ \ and \ $d=\de s>0$. \ If \ $K\subset \C^N$ \ is a compact set with norming sets $\{\aol_n\}_n$ and \ $E$ is given by (\ref{E jako podniesienie K}) \ then the finite sets 
	\be \label{B_n} \bol _n= \{(z,y)\in V\: : \: z\in \aol_{\ell} \} \ \ \ \ \mbox{where} \ \ \ell=\ell(n)=\min\left\{\!\!\left\lceil\!\frac{d^2n}k\!\right\rceil\!\!,\, dn+(k-1)(d-k)\!\!\right\}, \ \ \  n\in \N,
	\ee
	are norming for $E$. \  Moreover, if $C_n$ is a constant for $\aol_n$ on $K$ in inequality (\ref{mesh}), then \ $k\ C_{\ell(n)}^{1/k}$ is a~constant for $\bol_n$ on $E$ \ and \ $\ca \bol_n=k\: \ca \aol_{\ell(n)}$. \ In particular, for $d=k$ we have $\ell(n)=dn$. 
\end{thm}

The case $\de \varsigma_0=0$ is trivial:  $\varsigma_0=c$, $E=K\times \{\sqrt[k]{c}\}$ and for any polynomial $p\in\pol_n(z,y)$ 
\[ \|p\|_E = \|p(z,\sqrt[k]{c})\|_K\le C_n \|p(z,\sqrt[k]{c})\|_{\mathcal{A}_n} = C_n \|p\|_{\bol_n} \ \mbox{ where }\ \bol_n=\aol_n\times \{\sqrt[k]{c}\}.\]

\begin{proof}
Fix an arbitrary polynomial \ $p\in\pol_n(z,y)$ \ and consider \ $q\in\pol(z)\otimes \pol_{k-1}(y)$ \ such that $q\equiv p$ on $V$. By Theorem \ref{deg ze starej pracy}, $m:=\de q \le \frac{dn}k$. Write \ $q(z,y)=\sum_{j=0}^{k-1} q_j(z)\, y^j$ \ and for some $(z,y)\in E$ we have
	\be \label{t1} \|p\|_E = \|q\|_E = |q(z,y)| \le \sum_{j=0}^{k-1} |q_j^k(z)\, y^{jk}|^{1/k} = \sum_{j=0}^{k-1} |q_j(z)^k\, \varsigma_0(z)^j|^{1/k} \le \sum_{j=0}^{k-1} \|q_j^k\, \varsigma_0^j\|_K^{1/k}. \ee
	Consider $j$ such that $\|q_j^k\, \varsigma_0^j\|_K\not=0$. Observe that 
	\[ \de (q_j^k \varsigma_0^j) = k \: \de q_j + j \:\de \varsigma_0\le k(m-j) + j\:\de \varsigma_0 \le d(m-j)+jd =dm\le \frac{d^2n}k\] and 
	\[ \de (q_j^k \varsigma_0^j) \le k(m-j) + j\:\de \varsigma_0 = km+j(\de \varsigma_0-k)\le dn+(k-1)(d-k).	\]
	In both estimates, \ $ \de (q_j^k \varsigma_0^j) \le \ell(n)$ \ and since \ $\aol_{\ell(n)}$ is a norming set on $K$, we get
	\be \label{t2} \|q_j^k\, \varsigma_0^j\|_K \le C_{\ell(n)} \| q_j^k\, \varsigma_0^j \|_{\aol_{\ell(n)}} = C_{\ell(n)} \| q_j\, \varsigma_0^{j/k} \|_{\aol_{\ell(n)}}^k.\ee	
	For a fixed $z\in \aol_{\ell(n)}$ where $\| q_j^k\, \varsigma_0^j \|_{\aol_{\ell(n)}}$ is attained, take $\{y_1,...,y_k\}=\varsigma_0(z)^{1/k}\not=\{0\}$. Every $y_i$ is equal to $|\varsigma_0(z)|^{1/k}u_i$ where $\{u_1,...,u_k\}=\sqrt[k]1$. By Newton's identities and Vi\`ete's formulas, we have
	\[ \sum_{m=1}^k u_m ^\nu =0 \ \ \ \ \mbox{for } \ \nu=1,..., k-1 . \]
Consequently, 
\begin{align*}
k \: \| q_j\, \varsigma_0^{j/k} \|_{\aol_{\ell(n)}} &= k\: \left| q_j(z)\, \varsigma_0(z)^{j/k}\right|=   \left|q_j(z)\, |\varsigma_0(z)|^{j/k} \sum_{m=1}^k u_m ^k \right|\\ &= \left|\sum_{i=0}^{k-1} q_i(z) \, |\varsigma_0(z)|^{i/k} \sum_{m=1}^k u_m ^{k-j+i} \right|= \left|\sum_{m=1}^k u_m^{k-j} \sum_{i=0}^{k-1} q_i(z) \, |\varsigma_0(z)|^{i/k}  u_m ^{i} \right|\\ &\le \sum_{m=1}^k \left| \sum_{i=0}^{k-1} q_i(z) \, |\varsigma_0(z)|^{i/k}  u_m ^{i} \right|= \sum_{m=1}^k \left| \sum_{i=0}^{k-1} q_i(z) \, y_m ^{i} \right| \le \sum_{m=1}^{k} \|q\|_{\bol_n}\\ &= k\: \|q\|_{\bol_n} = k \: \|p\|_{\bol_n}.
\end{align*}
	From this and inequalities (\ref{t1}), (\ref{t2}) we obtain
	\[ \|p\|_E \le \sum_{j=0}^{k-1} \|q_j^k\, \varsigma_0^j\|_K^{1/k} \le C_{\ell(n)}^{1/k} \sum_{j=0}^{k-1} \| q_j\, \varsigma_0^{j/k} \|_{\aol_{\ell(n)}} \le C_{\ell(n)}^{1/k} \sum_{j=0}^{k-1} \|p\|_{\bol_n} = k \ C_{\ell(n)}^{1/k} \: \|p\|_{\bol_n}\]
	and the proof is complete.
\end{proof}

\vskip 2mm

\begin{corol} \label{CorOptAdm1}
	Under the assumptions and notations of Theorem \ref{yk=s0}, we have: 
	\begin{itemize}
		\item if $\{\aol_n\}_n$ are weakly admissible meshes on  $K$  then so are  $\{\bol_n\}_n$ on $E$,
		\item if $\{\aol_n\}_n$ are admissible meshes on  $K$  then so are  $\{\bol_n\}_n$ on $E$,
		\item if $\{\aol_n\}_n$ are optimal polynomial meshes on  $K$  then so are  $\{\bol_n\}_n$ on $E$.
	\end{itemize} 
	\end{corol}

\vskip 2mm

Regarding the converses, we can see that if $\{\bol_n\}_n$ are norming sets on $E\subset V(s)$, \ $\pi$ is the projection 
\begin{equation}\label{projection}
V(s)\ni(z,y)\mapsto z\in\C^N
\end{equation}
 and $K=\pi(E)$, then for any $p\in\pol_n(z)$ we have
\[ \|p\|_K = \|p\|_E \le C_n\: \|p\|_{\bol_n}= C_n\: \|p\|_{\pi(\bol_n)}. \] 

\vskip 2mm 

\begin{corol} \label{CorOptAdm2}
	Under the assumptions of Theorem \ref{yk=s0},  
	\begin{itemize}
		\item $E\subset V(s)$ has WAM if and only if 
		$K=\pi(E)$ has WAM,
		\item $E\subset V(s)$ has AM if and only if 
		$K=\pi(E)$ has AM,
		\item $E\subset V(s)$ has OPM if and only if 
		$K=\pi(E)$ has OPM.
	\end{itemize}
Moreover, meshes $\{\bol_n\}_n\subset E$ and $\{\aol_n\}_n\subset K$ are related by (\ref{B_n}) or $\pi(\bol_n)=\aol_n$.   
\end{corol}

\vskip 2mm

\subsection{On some algebraic sets of codimension greater than one} 
As before, we consider an algebraic variety $V=V(s_1,s_2)\subset\C^{N+2}$ defined by polynomials $s_1$ and $s_2$ of the form (\ref{s1yk}) and (\ref{s2yk}). Let $K$ be a compact subset of $\C^N$ and 
\begin{equation}\label{Ecodim2}
E=\{(z,y_1,y_2)\in V(s_1,s_2)\: :\: z\in K\}.
\end{equation}
We will construct meshes for $E$ using norming sets in $K$. We first assume that $\tilde{\varsigma}_1=\ldots=\tilde{\varsigma}_{k_1-1}=0$ and $\tilde{\tilde{\varsigma}}_1=\ldots=\tilde{\tilde{\varsigma}}_{k_2-1}=0$. Then 
both polynomials $s_1$ and $s_2$ have a specific form considered in the previous section and we can use Theorem \ref{yk=s0}. 

\begin{thm} \label{TwiceSpecific}
Let $V=V(s_1,s_2)\subset \C^{N+2}$ be an algebraic variety defined by polynomials of the form \ $s_1(z,y_1)=y_1^{k_1}-\tilde{\varsigma}_0(z)$, $\de \tilde{\varsigma}_0> 0$ \ and \ $s_2(z,y_1,y_2)=y_2^{k_2}-\tilde{\tilde{\varsigma}}_0(z,y_1)$, $\de \tilde{\tilde{\varsigma}}_0> 0$ with $d_1=\de s_1$, $d_2=\de s_2$. \ If \ $K\subset \C^N$ \ is a compact set with norming sets $\{\aol_n\}_n$, $C_n$ is a constant for $\aol_n$ on $K$ in inequality (\ref{mesh}) and \ $E$ is given by (\ref{Ecodim2}) \ then the finite sets 
	\be \label{B_ncodim2} \bol _n= \{(z,y_1,y_2)\in V\: : \: z\in \aol_{\ell(n)} \}, \ \ \ n\in\N \ee 
where
\[ \ell(n)=d_1d_2n+d_1d_2k_2+d_1k_1 \]
are norming for $E$ with constants $k_2k_1^{1/k_2}\ C_{\ell(n)}^{1/(k_1k_2)}$ and $\ca \bol_n=k_1k_2\: \ca \aol_{\ell(n)}$. Moreover, for $d_1=k_1$ and $d_2=k_2$ we have $\ell(n)=k_1k_2n$. 
\end{thm}
\begin{proof} Let $V$, $E$ and $K$ satisfy the assumptions. We define $F:=\{(z,y_1)\in V(s_1)\: :\: z\in K\}$ and $V(s_1)=\{(z,y_1)\in \C^N\times \C \: : \:s_1(z,y_1)=0\}$. The set $F$ is a compact subset of $\C^{N+1}$ and from Theorem \ref{yk=s0} the sets 
\[ \fol _n= \{(z,y_1)\in V(s_1)\: : \: z\in \aol_{j(n)} \}, \ \ \ n\in \N \]
are norming sets for $F$ with constants $\tilde{C_n}=k_1 C_{j(n)}^{1/k_1}$ where \[j(n)=d_1n+\min\left\{\!\!\left\lceil\!\frac{d_1(d_1-k_1)n}{k_1}\!\right\rceil\!\!,\,(k_1-1)(d_1-k_1)\!\!\right\}<d_1(n+k_1).\]

Additionally, $E=\{(z,y_1,y_2)\in V(s_2)\: :\: (z,y_1)\in F\}$ where $V(s_2)=\{(z,y_1,y_2)\in \C^{N+2}\: : \:s_2(z,y_1,y_2)=0\}$.
By Theorem \ref{yk=s0} again, 
\[
\bol _n= \{(z,y_1,y_2)\in V(s_2)\: : \: (z,y_1)\in \fol_{t(n)} \}, \ \ \ n\in\N\]
are norming sets for $E$ where $t(n)=d_2n+\min\left\{\!\!\left\lceil\!\frac{d_2(d_2-k_2)n}{k_2}\!\right\rceil\!\!,\,(k_2-1)(d_2-k_2)\!\!\right\}<d_2(n+k_2)$.
Finally, we obtain
\[
\bol _n= \{(z,y_1,y_2)\in V(s_1,s_2)\: : \: z\in \aol_{j(t(n))} \}.
\]
\end{proof}

\begin{corol} \label{CorOptAdmCodim}
	Under the assumptions and notations of Theorem \ref{TwiceSpecific}, we have: 
	\begin{itemize}
		\item if $\{\aol_n\}_n$ are weakly admissible meshes on  $K$  then so are  $\{\bol_n\}_n$ on $E$,
		\item if $\{\aol_n\}_n$ are admissible meshes on  $K$  then so are  $\{\bol_n\}_n$ on $E$,
		\item if $\{\aol_n\}_n$ are optimal polynomial meshes on  $K$  then so are  $\{\bol_n\}_n$ on $E$.
	\end{itemize} 
	\end{corol}

We now turn back to an algebraic variety  $V=V(s_1,s_2)\subset\C^{N+2}$ defined by polynomials $s_1$ and $s_2$ of the form (\ref{s1yk}) and (\ref{s2yk}). On a compact set $E$ given by (\ref{Ecodim2}) we can get weakly admissible mesh. We start with the observation concerning the division inequality. We need
a property that was proved in \cite{BCCK2}, 

\begin{proposition} \label{war_geom}
Let $V(s)\subset\mathbb{C}^{N+1}$ be an algebraic variety defined by a polynomial $s$ in the form $s(z,y) = y^k-\sum\limits_{j=0}^{k-1}s_j(z)\,y^j$ with $k\ge 1$, $K\subset\mathbb{C}^N$ be a compact set and  $E:=\pi^{-1}(K)\subset V(s)$ where $\pi$ denotes the projection (\ref{projection}).
If $K$ satisfies the division inequality with exponent $m$ and \ $K\setminus\mathcal{A}(s)\neq\emptyset$ \ then 
\begin{equation} \label{skladniki}
\|[p_0,...,p_{k-1}]\|_K  \le \ M_0\, (\deg p)^{m_0} \|p\|_{E} \end{equation} 
for any polynomial $p$ written in the form \ $p(z,y)=\sum\limits_{j=0}^{k-1} p_j(z)\, y^j$ on $V(s)$ where $M_0,\, m_0\ge 0$ are constants independent of $p_0,\ldots,p_{k-1}$  and \ $m_0=m\,(k-1)\: \deg s$.
\end{proposition}
	
\begin{thm} \label{th:V-div2}
Let $V=V(s)\subset \C^{N+1}$ be an algebraic hypersurface given by an~irreducible polynomial $s$ in form (\ref{syk}) with $\deg s=d$ and $K$ be a compact subset of $\C^N$ and $F:=\{(z,y)\in V\: : \: z\in K\}$. If 
$K$ is  determining for polynomials from $\mathcal{P}(z)$ and satisfies the division inequality then $F$ also satisfies the division inequality.
\end{thm}	

\begin{proof} Fix a polynomial $q\in\mathcal{P}(z,y)$, $q_{|_F}\not\equiv0$ and 
find $\mathfrak{q}\in\pol(z) \otimes \pol_{k-1}(y)$ such that $\mathfrak{q}=q$ on $V(s)$ and write it in the form $\mathfrak{q}(z,y)=\sum\limits_{j=0}^{k-1}q_j(z)y^j$. 
Fix also a polynomial $p\in \mathcal{P}(z,y)$ and find $\mathfrak{p}\in\pol(z) \otimes \pol_{k-1}(y)$ such that $\mathfrak{p}=p$ on $V(s)$ and $\mathfrak{p}(z,y)=\sum\limits_{j=0}^{k-1} p_j(z)\, y^j$. From Lemma \ref{lemat ze stopniami} we have 
\[
\deg \mathfrak{q}\le d\deg q \ \mbox{ and }\ \deg \mathfrak{p}\le d\deg p.
\] Let \ $\mathbf{P}:=[p_0,...,p_{k-1}]^T$. 
Consider the matrix $\mathbf{M}_{\mathfrak{q}}^s$ such that for \ $(z,y)\in V$ we have
\[ \mathfrak{p}(z,y)\: \mathfrak{q}(z,y) \ = \ [1,y,...,y^{k-1}] \ \mathbf{M}_{\mathfrak{q}}^s \ \mathbf{P}(z).\] The determinant of $\mathbf{M}_{\mathfrak{q}}^s$ \ is a polynomial only in $z$ and is equal to the resultant Res$\,_y(\mathfrak{q},s)$. Since $q$ is coprime with $s$, Res$_y\,(\mathfrak{q},s)$ is a non-zero polynomial in $\mathbb{C}^N$ and so is $\det \mathbf{M}_{\mathfrak{q}}^s$. Since $K$ is determining for polynomials from $\mathcal{P}(z)$ we have $\det \mathbf{M}_{\mathfrak{q}}^s\not\equiv0$ on $K$. Observe that 
\[\deg \, \det\, \mathbf{M}_{\mathfrak{q}}^s \le 
k\deg \mathfrak{q} + \tfrac{(k-1)kd}2 \le kd \, \deg q  + \tfrac{(k-1)kd}2 = kd \, (\deg q +\tfrac{k-1}2)\]
Thanks to division inequality on $K$, we have
$$ \|p\|_E \le \sum_{j=0}^{k-1} \|p_j\|_{K} \, \|y\|_E^j \le M\sum_{j=0}^{k-1} (\deg p_j+kd \, (\deg q +\tfrac{k-1}2))^{mkd \, (\deg q +\tfrac{k-1}2)} \|\det \mathbf{M}_{\mathfrak{q}}^s p_j\|_{K} \|y\|_E^j$$
\[
\le M (\deg p+dk^2 \, \deg q)^{mdk^2 \, \deg q} \|(\det \mathbf{M}_{\mathfrak{q}}^s) \mathbf{P}\|_{K} \sum_{j=0}^{k-1} \|y\|_E^j.
\]
Replacing
$(\det \mathbf{M}_{\mathfrak{q}}^s) \mathbf{P}$ by $\mathbf{A}\mathbf{M}_{\mathfrak{q}}^s \mathbf{P}$, where $\mathbf{A}$ denote the transpose of the comatrix of $\det\, \mathbf{M}_{\mathfrak{q}}^s$, we obtain
\[
\|p\|_E\le M (\deg p+dk^2 \, \deg q)^{mdk^2 \, \deg q} \|\mathbf{A}\|\|\mathbf{M}_{\mathfrak{q}}^s \mathbf{P}\|_{K} \sum_{j=0}^{k-1} \|y\|_E^j
\]
Moreover, 
\[\|\mathbf{M}_{\mathfrak{q}}^s \: \mathbf{P}\|_{K}=\|[u_0,...,u_{k-1}]\|_{K} \mbox{ where } (\mathfrak{p}\mathfrak{q})(z,y)=\sum\limits_{j=0}^{k-1} u_j(z)\, y^j \mbox{ on }V(s).\]
By Proposition \ref{war_geom},
\[\|[u_0,...,u_{k-1}]\|_{K} \le M_0 (\deg(\mathfrak{p}\mathfrak{q}))^{m_0} \|\mathfrak{p}\mathfrak{q}\|_E \le M_0 d^{m_0}(\deg p +  \ \deg q)^{m_0} \|pq\|_E\]
and the proof is complete.
\end{proof}

Since $E=\{(z,y_1,y_2)\in V(s_2)\: :\: (z,y_1)\in F\}$ and from Theorem \ref{th:V-div2} $F$ satisfies the division inequality, we can apply Theorem \ref{ThmAdmMesh} twice to obtain

\begin{thm}\label{ThmAdmMesh2} Let $V=V(s_1,s_2)\subset\C^{N+2}$ be an algebraic variety defined by polynomials $s_1$ and $s_2$ of the form (\ref{s1yk}) and (\ref{s2yk}), $s_1$ be an~irreducible polynomial, $K\subset \mathbb{C}^N$ be a compact set with weakly admissible meshes $(\mathcal{A}_n)_{n\in\mathbb{N}}$ and $F:=\{(z,y_1)\in V(s_1)\: :\: z\in K\}$. If $K$ is  determining for polynomials from $\mathcal{P}(z)$ and satisfies the division inequality and $K\setminus\mathcal{A}(s_1)\neq\emptyset$ and $F\setminus\mathcal{A}(s_2)\neq\emptyset$ then \[E:=\{(z,y_1,y_2)\in V(s_1,s_2)\: :\: z\in K\}\] has  weakly admissible meshes $(\mathcal{B}_n)_{n\in\mathbb{N}}$ defined by
\[\mathcal{B}_n=\{(z,y_1,y_2)\in E\ :\ z\in \mathcal{A}_l\cup\{a,b\}\}\]
where $a$ is an element of $K\setminus\mathcal{A}(s_1)$, $b$ is an element of $\pi(F\setminus\mathcal{A}(s_2))$ and
\[l=l(n)=\left\{\begin{array}{rl} 4n+4(k_2^2-k_2)+2(k_1^2-k_1)& \mbox{ if }d_1=k_1 \mbox{ and }d_2=k_2,\\
4d_1n+4d_1(k_2^2-k_2)& \mbox{ if }d_1>k_1 \mbox{ and }d_2=k_2,\\
4d_2n+2(k_1^2-k_1)& \mbox{ if }d_1=k_1 \mbox{ and } d_2>k_2,\\
4d_1d_2n & \mbox{ if }d_1>k_1 \mbox{ and }  d_2>k_2 .
\end{array}\right.\]
More precisely, for every polynomial $P\in\pol_n(\mathbb{C}^{N+2})$ we have
\[\|P\|_E\le C(\mathcal{B}_n) \|P\|_{\mathcal{B}_n}\] 
and $(C(\mathcal{B}_n))_n$ are of polynomial growth in $n$.
\end{thm}


\vskip 5mm

\section{Applications}

\subsection{Classical results} We start this section with three well known results, see e.g., \cite{BV}, to recall optimal polynomial meshes on a segment, on a disk in $\C$ and on a disk in $\R^2\subset\C^2$. 

\begin{thm} \label{mesh na ab}
	Let $[a,b]$ be a segment in $\C$.	
	For any polynomial $p\in \pol_n(\C)$ and any integer number $\lambda>n$ we have
	\[ \|p\|_{[a,b]} \ \le \ \frac1{\cos\frac{n\pi}{2\lambda}} \: \|p\|_{\mathcal{M}_{\lambda,n}} \]
	where \ $\mathcal{M}_{\lambda,n} =\left\{ \frac{a+b}2 + \frac{a-b}2\: \cos\frac{(2j-1)\pi}{2\lambda} \: : \: j=1,...,\lambda\right\}$ \ 
	is the scaled $\lambda$-system of the Chebyshev points, i.e., zeros of the Chebyshev polynomial $T_\lambda(x)=\cos(\lambda \arccos x)$.  
	In particular, for $\lambda =\left\lceil \frac32n \right\rceil$, $n\ge 2$ the set $\mathcal{M}_{\lceil \frac32n \rceil,n}$ is an optimal polynomial mesh for polynomials from $\pol_n(\C)$ and inequality (\ref{mesh}) holds with the constant $C_n=2$. 
\end{thm}

\begin{thm} \label{mesh na disk in C}
	Let $D(a,r)$ be a disk in $\C$ with radius $r$ and centre at $a$.	
	For any polynomial $p\in \pol_n(\C)$ and any integer number $\lambda>n$ we have
	\[ \|p\|_{D(a,r)} \ \le \ \frac1{\cos\frac{n\pi}{2\lambda}} \: \|p\|_{\mathcal{M}_{\lambda,n}} \]
	where \ 
	$\mathcal{M}_{\lambda,n} =\left\{ a+r\: \exp\left(i\: \frac{j\pi}{\lambda}\right) \: : \: j=1,...,2\lambda\right\}=a+r\sqrt[2\lambda]1$ \
	is the set of equidistributed $2\lambda$ points on the circle $C(a,r)$.    
\end{thm}

\begin{thm} \label{mesh na disk in R2}
	Let $B(a,r)$ be a disk in $\R^2\subset \C^2$ with radius $r$ and centre at $a$.	
	For any polynomial $p\in \pol_n(\C^2)$ and any integer number $\lambda>n$ we have
	\[ \|p\|_{B(a,r)} \ \le \ \frac1{\cos^2\frac{n\pi}{2\lambda}} \: \|p\|_{\mathcal{M}_{\lambda,n}} \]
	where \
	$\mathcal{M}_{\lambda,n} =\left\{ a+ \left(r_j \cos \frac{m\pi}\lambda,\  r_j \sin \frac{m\pi}\lambda\right) \right\}_{j,m=1,...,\lambda}, \ \	r_j=r\: \cos\frac{(2j-1)\pi}{2\lambda}.$
\end{thm}

If we choose $\lambda$ such that $\tfrac{n}\lambda$ can be estimated from above by a constant independent of $n$, then we obtain optimal polynomial meshes in the three above theorems. If not, we have weakly admissible meshes. 

\subsection{On algebraic hypersurfaces} We give here three quite simple examples of applications of our results concerning algebraic sets given by one algebraic equation $s(z,y)=0$. Two examples are chosen so that the compact sets under consideration are in a real space. 

\begin{example} 
We firstly give optimal polynomial meshes on the unit sphere in $\R^3$:
\[ S_2=\{(x,y,z)\in \R^3\: : \: x^2+y^2+z^2=1\}. \]
Let $V=\{(x,y,z)\in \C^3\: : \: x^2+y^2+z^2=1\}$. We have $s(x,y,z)=x^2+y^2+z^2-1$ and $d=k=2$. 
In the set $K=\{(x,y)\in \R^2\: : \: x^2+y^2\le 1\}$ we take the optimal polynomial meshes $\{\aol_{n}\}_n$
constructed for $\lambda=2n$ in Theorem \ref{mesh na disk in R2}, that is $\aol_{n}=\mathcal{M}_{2n,n}$. By Theorem \ref{yk=s0}, $\ell(n)=2n$ and thus
\[ \mathcal{B}_n=\{(x,y,z)\in V\: : \: (x,y)\in\aol_{2n}\} \]
\[ = \{(x,y,z)\in \R^3 \: :\: x=r_j\cos \frac{m\pi}{4n}, \: y=r_j\sin \frac{m\pi}{4n}, \: x^2+y^2+z^2=1, \: j,m=1,...,4n \}  \]
where $r_j=\cos\frac{(2j-1)\pi}{8n}$.
Consequently, for any polynomial $p\in\pol_n(\C^3)$
\[ \|p\|_{S_2}\le 2 C_{\ell(n)}^{1/2}\|p\|_{\bol_n} = \frac{2}{\cos \frac{\ell(n)\pi}{2\lambda}} \|p\|_{\bol_n} = \frac2{\cos\frac\pi4}\|p\|_{\bol_n}=2\sqrt{2}\|p\|_{\bol_n} \]
and card\,$\bol_n=2\mbox{card}\aol_{\ell(n)}=2\dot(4n)^2=32n^2$, cf. \cite[Corollary 1]{LSV}.

	In a similar way, we can construct admissible meshes on a spherical section, cf. \cite{SV18}. Let $E=\{(x,y,z)\in V\: : \: (x,y)\in K\}$ where $K=\{(x,y)\in \R^2\: : \: x=\varrho \cos \varphi, \ y= \varrho \sin \varphi, \ \varrho \in [0,1], \ \varphi\in [0,\alpha]\}$, \ $\alpha\in (0,\pi]$. Observe that $E$ is a subset of $\R^3$. Using norming sets constructed in \cite{SV18} for  sections of a disk, by Theorem \ref{yk=s0} we can obtain optimal admissible meshes on a spherical lune $E$. 
\end{example}

\begin{example}\label{Ex5}
	Now consider the algebraic hypersurface 
	\[V=\{(x,y,z)\in \C^3\: : \: z^2 -x^3 -y^2=0\}\]
	and its compact subset 
	\[E=\{(x,y,z) \in V\: : \: (x-1)^2+y^2\le 1, \ x,y\in \R\}.\]
Observe that \ $E\subset\R^3$. 
The projection of $E$ into the space of two complex variables $(x,y)$ yields the set $K=\{(x,y)\in \R^2\: : \: (x-1)^2+y^2\le 1\}.$ 
The optimal polynomial meshes $\{\aol_n\}_n$ on $K$ are given in Theorem \ref{mesh na disk in R2}. Indeed, the family $\mathcal{A}_\ell =\mathcal{M}_{\lambda,\ell}$ with $\lambda=\left\lceil \tfrac{4\ell}3 \right\rceil$ is an optimal polynomial mesh in $K$, since $\tfrac{\ell}\lambda=\tfrac{\ell}{\left\lceil \tfrac{4\ell}3 \right\rceil} \le \tfrac34$.
In particular, it can be seen that, independently of $\ell$, we have $\|p\|_K \leq 6.8285 \|p\|_{\mathcal{A}_\ell}$, for any $p \in {\mathcal{P}}_\ell(K)$.

	For the polynomial $s(x,y,z)=z^2-x^3-y^2$, we have $d=\de s=3$, $k=2$. By Theorem \ref{yk=s0}, 
    $\bol_n=\{(x,y,z)\in V\: : \: (x,y)\in \aol_{\ell(n)}\}$ is an optimal polynomial mesh in $E$ where $\ell=\ell(n) = 3n+1 $ for any $n$. More precisely, $\lambda=\left\lceil \tfrac{4\,\ell(n)}3\right\rceil=\left\lceil \tfrac{4(3n+1)}3\right\rceil=4n+2$ and 
	$$ \bol_n=\{(x,y,z)\: : \: x=1+r_j\cos \frac{m\pi}{4n+2}, \ y=r_j\sin\frac{m\pi}{4n+2}, \ z^2 -x^3 -y^2=0, \ j,m=1,...,4n+2\}, $$
	$r_j=\cos\frac{(2j-1)\pi}{8n+4}$.
	We can easily see that $\ca \bol_n=2 \,(4n+2)^2=\mathcal{O}(n^2)=\mathcal{O}({\rm dim}\,\pol_n(E))$ and $C_n=k\, C_{\ell(n)}^{1/k}= \frac2{\left(\cos^2\frac{\ell(n)\pi}{2\lambda}\right)^{1/2}}= \frac2{\cos\frac{(3n+1)\pi}{8n+4}}\le \frac2{\cos\frac{3\pi}8} \approx 5.2263$. Finally, for every polynomial $p\in\pol_n(\C^3)$ we have
	\[ \|p\|_E \ \le \ 5.23 \: \|p\|_{\bol_n}\]
	and $\bol_n$ are optimal polynomial meshes on $E\subset V(s)$.

    In Figures {\ref{fig:FigExa5_1}} and {\ref{fig:FigExa5_2}}, we show respectively the norming sets $\aol _7={\mathcal{M}}_{10,7}$ and ${\bol _2}$ providing details about their cardinalities.

	\begin{figure}[!h]
        \includegraphics[scale=0.3]{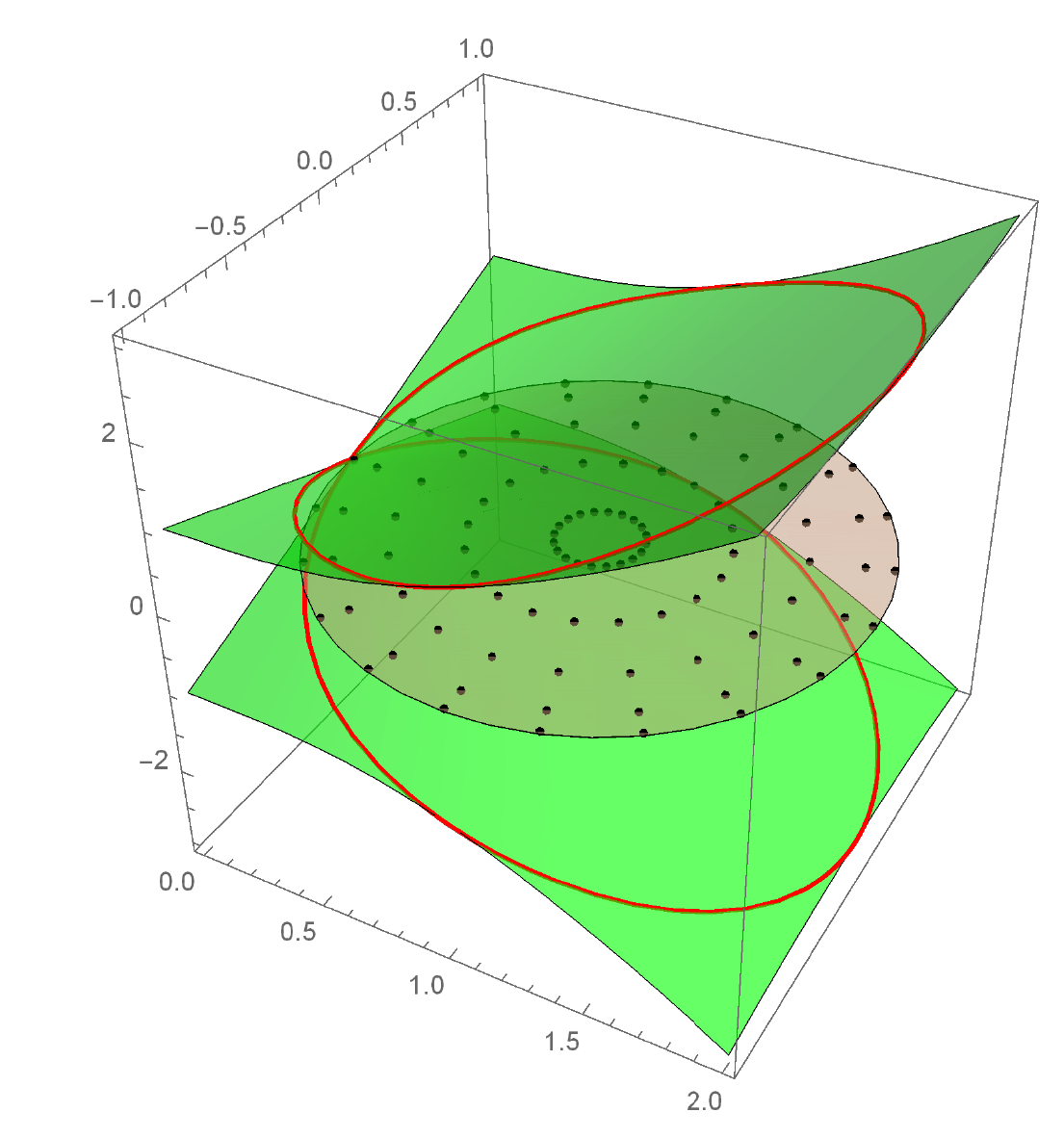}
		\caption{The curves in red represent the boundary of the set $E$ in Example \ref{Ex5}. The black dots are the points of the disk $K$ that define the norming set $\aol _7={\mathcal{M}}_{10,7}$, used to determine  $\bol _2$ on $E$. In particular, $\ca \aol_7=100$ and  $\|p\|_K \leq 4.86 \|p\|_{\aol _n}$, $p \in {\mathcal{P}}_7(K)$.}
		\label{fig:FigExa5_1}
	\end{figure}
	
    \begin{figure}[!h]
        \includegraphics[scale=0.3]{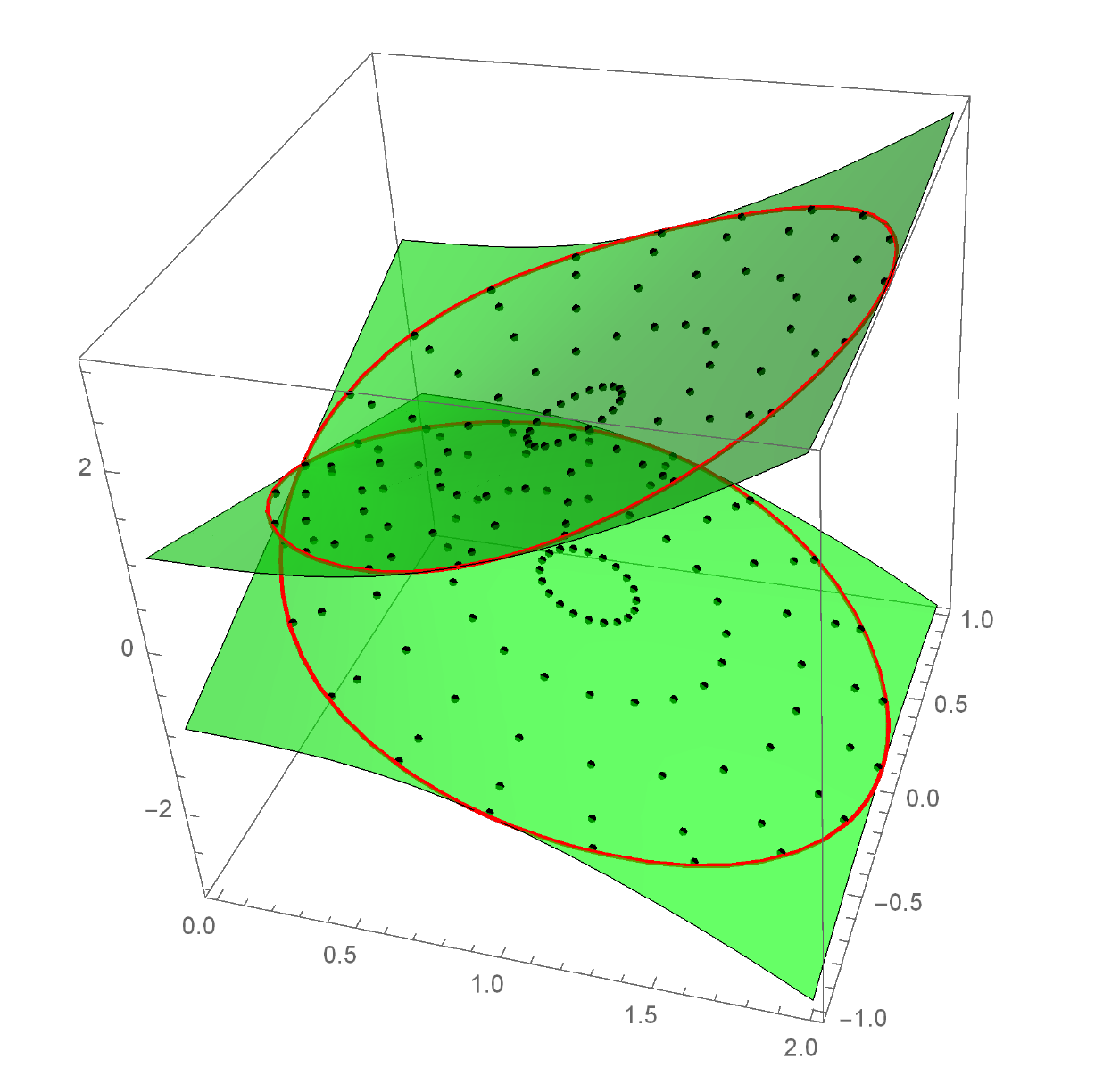}
		\caption{The black dots are the points of the norming set $\bol _2$ for $E$ in Example \ref{Ex5}. In particular, $\ca \bol_2=200$ and for any $n$ we have $\|p\|_E \leq 5.23 \|p\|_{\bol _n}$, $p \in {\mathcal{P}}_n(E)$.}
		\label{fig:FigExa5_2}
	\end{figure}	
	 
\end{example}  

\begin{example}
	In this example we construct optimal polynomial meshes on the set
	\[E=\{(x,y)\in \C^2 \: : \: |x|\le 1, \ y^3-x^2+1=0\}\]
	that is a compact subset of the algebraic hypersurface $V=\{(x,y)\in \C^2 \: : \: y^3-x^2+1=0\}$ \ over the unit disk, i.e. $K=D(0,1)$. We have $s(x,y)=y^3-x^2+1$, $d=3=k$ and by means of Theorem \ref{yk=s0}, we obtain $\ell(n)=3n$.  
    If ${\mathcal{M}}_{\lambda,n}$ are the norming sets defined in Theorem \ref{mesh na disk in C}, setting ${\mathcal{A}}_{n}={\mathcal{M}}_{\lambda,n}$ with $\lambda=\left\lceil\frac43\ell(n)\right\rceil=4n$, we get that the family $\{{\mathcal{A}}_{n}\}_{n}$  is actually an optimal admissible mesh where $C_n = \frac{1}{\cos{\frac{\pi}{8}}} \approx 1.0824$. Thus if
	\[ \bol_n=\{(x,y)\in V\: : \: x\in \aol_{3n}\}=\{(x,y)\in V\: : \: x^{8n}=1\}, \ \ \ \ \ n\in \N\]
	the family $\{\bol_n\}_n$ is an optimal admissible mesh over $E$ because $\ca \bol_n=24n=\mathcal{O}({\rm dim}\,\pol_n(E))$ and for all $p\in\pol_n(\C^2)$ we have
	$$ \|p\|_E \le 3 C_{\ell(n)}^{1/3} \|p\|_{\bol_n} = 3 C_{3n}^{1/3} \|p\|_{\bol_n} = \frac3{\left(\cos\frac{3n\pi}{8n} \right)^{1/3}} \|p\|_{\bol_n}  = \frac3{\left(\cos\frac{3\pi}{8} \right)^{1/3}} \|p\|_{\bol_n} \le 3.1 \|p\|_{\bol_n}.
	$$ 
\end{example} 

\subsection{On algebraic sets of codimension greater than one}
We present here one example that are not directly related to the special case regarded in Theorem \ref{TwiceSpecific}. However, by means of Gr\"obner basis, we are able to write it in such a way that we can use this theorem. 

\begin{example}\label{VivWin}
Take $V=V(x^2+y^2+z^2-4,\ x^2+4y^2-z^2-10x+4)$. We can check that \[V=V(x^2+y^2+z^2-4,\ x^2+y^2-2x).\] Consider the set $K=[0,2]$ and $$E=\{(x,y,z)\in V\: :\: x\in K\}=\{(x,y,z)\in\C^3\: :\: x^2+y^2+z^2-4=0, \, x^2+y^2-2x=0,\ x\in[0,2]\}.$$ Observe that $E$ is a subset of $\R^3$. This real curve is called \textit{Viviani's window}. 

\begin{figure}[!h]
        \includegraphics[scale=0.3]{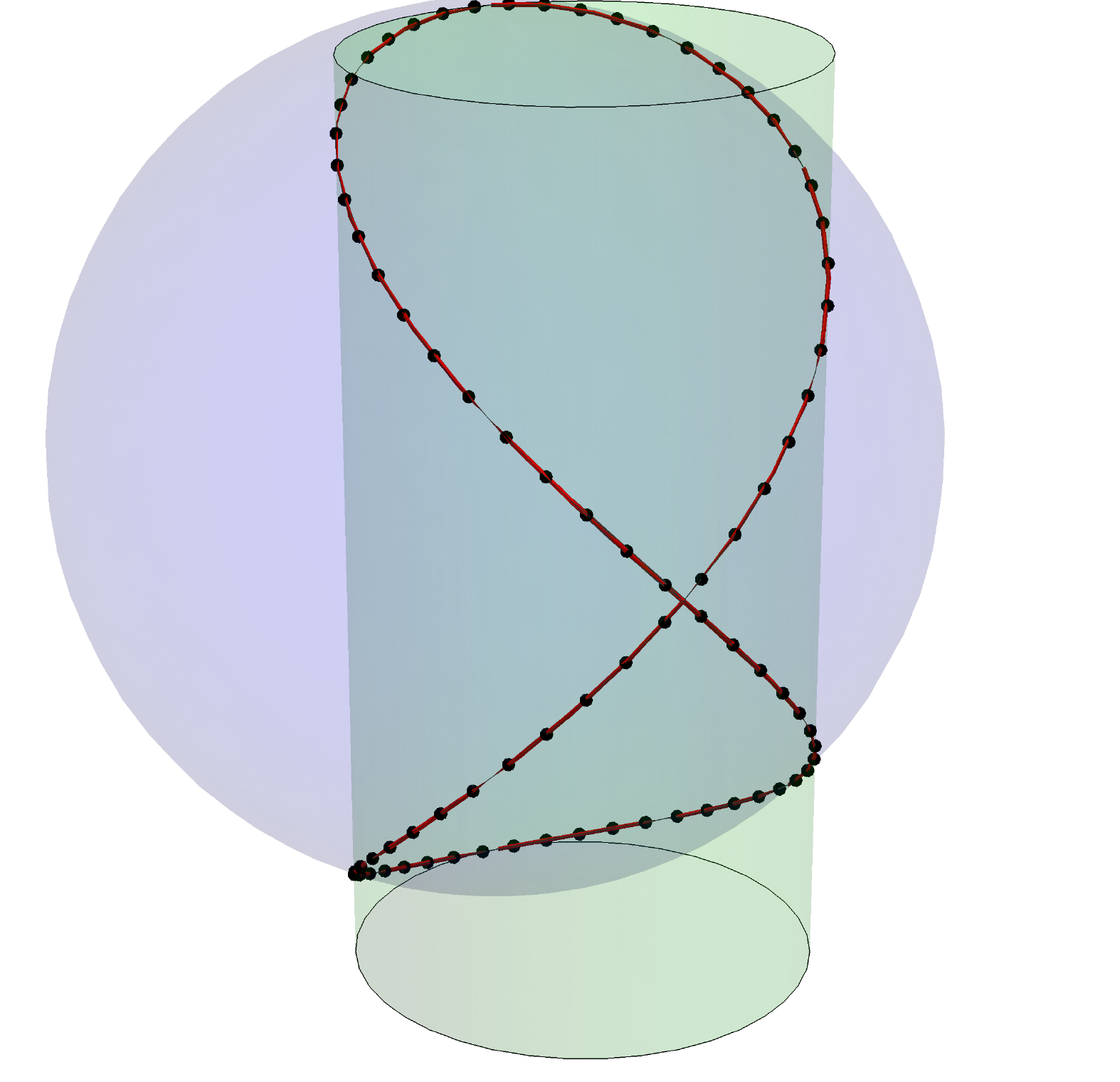}
		\caption{The norming set $\bol _4$ for Viviani's window $E$ in Example \ref{VivWin}, $\lambda=20$, $\ca \bol_4=80$ and $C(\bol_4)=3.8$.}
		\label{fig:FigExaViv1}
	\end{figure}
Following Theorem \ref{mesh na ab}, for $K=[0,2]$ and $\lambda > n$, the set
\[
{\mathcal{M}}_{\lambda,n} =\left\{ 1 + \cos\frac{(2j-1)\pi}{2\lambda} \: : \: j=1,...,\lambda\right\}.
\]  is such that 
$$
\|p\|_{[0,2]} \leq {\frac{1}{\cos \frac{n\pi}{2\lambda}}} \|p\|_{{\mathcal{M}}_{\lambda,n}}.
$$
If $\aol_n={\mathcal{M}}_{\lambda,n}$ for $\lambda=\lceil \frac{5n}{4}\rceil$, then the family $\{ \aol_n \}_n$ is an optimal admissible mesh since for each $n$
$$
\|p\|_K \leq \frac{1}{\cos{\frac{n\pi}{2\lceil \frac{5n}{4}\rceil}}}\|p\|_{\aol_n} \leq 3.2361 \|p\|_{\aol_n}.
$$

\begin{figure}[!h]
        \includegraphics[scale=0.3]{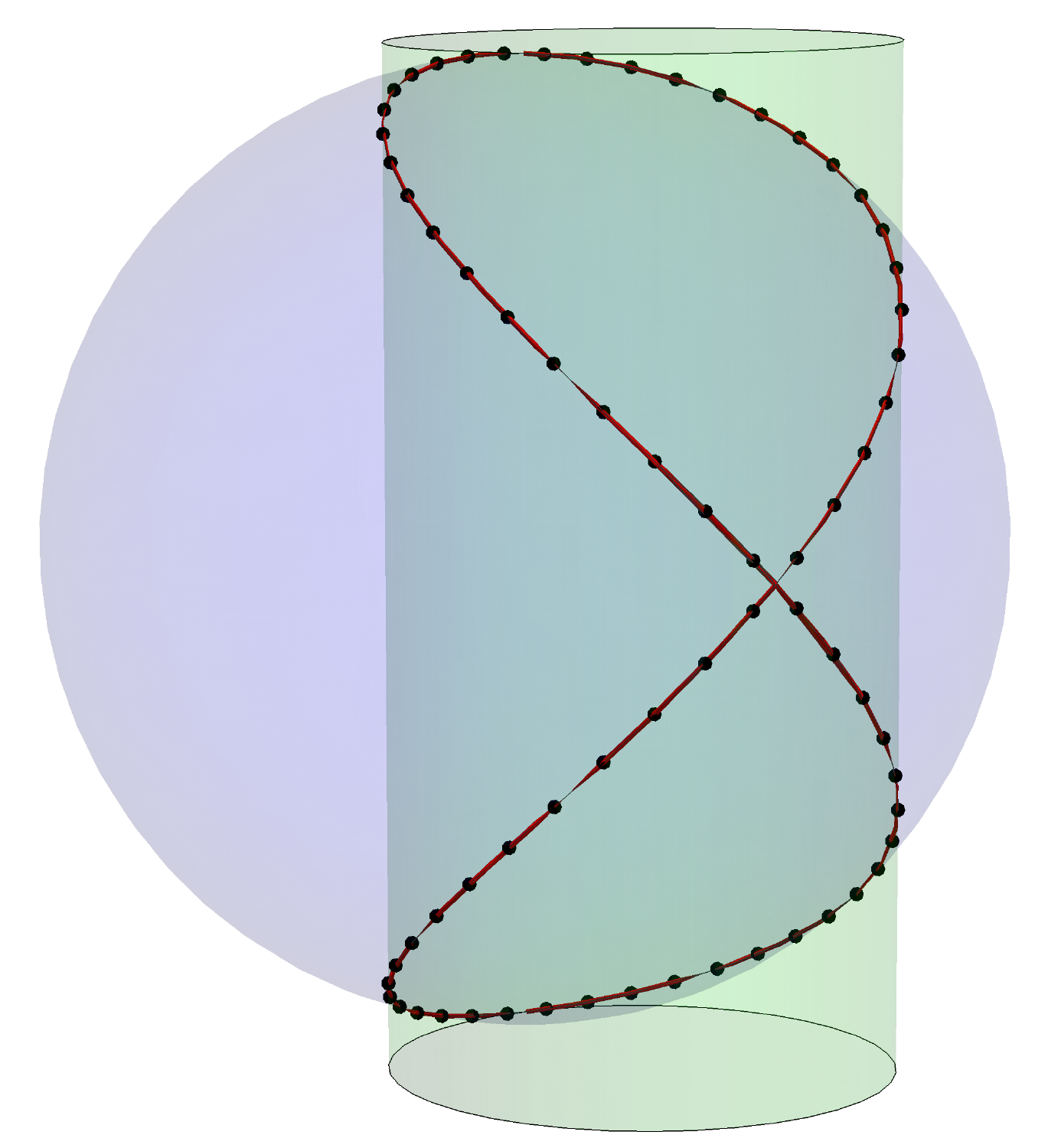}
		\caption{The norming set $\bol _4$ for Viviani's window $E$ in Example \ref{VivWin}, $\lambda=17$, $\ca \bol_4=68$ and $C(\bol_4)=5.2$.}
		\label{fig:FigExaViv2}
	\end{figure}

We have $s_1(x,y)=y^2+x^2-2x$ with $d_1=k_1=2$ and $s_2(x,y,z)=z^2+y^2+x^2-4$ with $d_2=k_2=2$. Clearly, $[0,2]$ is determining for polynomials from $\mathcal{P}(\mathbb{C})$. From Theorem \ref{TwiceSpecific} we obtain $\ell(n)=4n$ and consequently
$$ \bol _n= \{(z,y_1,y_2)\in V\: : \: z\in \aol_{4n}={\mathcal{M}}_{5n,4n}\}\!=\!\{(z,y_1,y_2)\in V\: : \: z=\!1 \!+ \cos\frac{(2j-1)\pi}{10n},\: j\!=\!1,\ldots,5n \}$$ 
is an optimal polynomial mesh on $E$. It can be easily seen that $\ca \bol_n=4\: \ca \aol_{4n}=20n$. Finally, for all $p\in\pol_n(\C^3)$ we have
	\[ \|p\|_E \le 2\sqrt{2} C_{4n}^{1/4} \|p\|_{\bol_n} = \frac{2\sqrt{2}}{\left(\cos\frac{4n\pi}{10n} \right)^{1/4}} \|p\|_{\bol_n}  = \frac{2\sqrt{2}}{\left(\cos\frac{2\pi}{5} \right)^{1/4}} \|p\|_{\bol_n} \le 3.8 \|p\|_{\bol_n}.\]

	\end{example}

\vskip 5mm

\section{Numerical examples}

In this section we consider the sets $E$ proposed in Example \ref{Ex5} and Example \ref{VivWin}, using the previously presented optimal admissible meshes $\bol_n=\{P_i\}_{i=1,\ldots,N_n}$ to determine good pointsets for interpolation as Approximate Fekete Points and Discrete Leja Points (shortened respectively with the acronyms AFP and DLP) as well as the classical least-squares on $\bol_n$, testing their performance. For a description of AFP and DLP, see e.g. \cite{BDMSV}.

All the routines are available as open source codes in {\cite{BCKS}} and where tested on an Intel Core Ultra 5 125H 3.60 GHz processor with 16GB of RAM, using Matlab 2024a.

We start our analysis from the set $E$ in Example {\ref{Ex5}}.
As test functions we adopt
\begin{eqnarray}
f_1(x,y,z) &=& (x+0.5 \, y +2 \, z+1)^{14};\nonumber \\
f_2(x,y,z) &=& \exp(-(x^2+0.5\,y^2+2\,z^2));\nonumber \\
f_3(x,y,z) &=& \sin(4\,x+5\,y+3z); \nonumber \\
f_4(x,y,z) &=& \sqrt{(x-1)^2+y^2+(z-1)^2}+\sqrt{(x-1)^2+y^2+(z+1)^2}. \nonumber 
\end{eqnarray}

For each $f=f_k$, $k=1,2,3,4$, after computing the interpolant/approximant ${\mathcal{L}}_n f$ for $n=2,4,\ldots,16$, we evaluate the errors $$\frac{\|f(\bol_{30})-{\mathcal{L}}_n f(\bol_{30}) \|_2}{\|f(\bol_{30})\|_2},$$ being $f(\bol_{30})=(f(P_i))_{i=1,\ldots,N_{30}}$ and ${\mathcal{L}}_n f(\bol_n) =({\mathcal{L}}_n f(P_i))_{i=1,\ldots,N_{30}}$.

The first three functions are regular, the third one having some oscillations, while the last one is the distance from the points $(1,0,1)$ and $(1,0,-1)$ that belong to the hypersurface.

As for the polynomial basis of ${\mathcal{P}}_n(E)$ we proceed as follows. We
\begin{itemize}
\item determine numerically, by means of the points in $X=\bol_n=\{P_i\}_{i=1,\dots,N_n} \subset E$, an approximation $[a_1,b_1] \times [a_2,b_2] \times [a_3,b_3]$  of the smaller Cartesian parallelepiped containing the set $E$;
\item being ${\eta}_n:=\mbox{dim}({\mathcal{P}}_n(E))$, evaluate the tensorial Chebyshev basis $\{\phi_{i}\}_{i=1,\ldots,{\eta}_n}$ of total degree $n$, shifted on $[a_1,b_1] \times [a_2,b_2] \times [a_3,b_3]$, on the norming set $X$, so determining the Chebyshev-Vandermonde-like matrix $U_X=(\phi_j(P_i))_{i,j} \in { \mathbb{R}}^{N_n \times {\eta}_n}$;
\item  compute the QR factorisation with column pivoting $U_X(:,\pi)=QR$, where $\pi$ is a permutation vector;
\item select the orthogonal matrix $V_X=Q(:,1:{\eta}_n)$; the first ${\eta}_n$ columns correspond to an orthonormal basis $\{\psi_{i}\}_{i=1,\ldots,{\eta}_n}$ of ${\mathcal{P}}_n(E)$, w.r.t. the discrete measure (X,{\bf{w}}), where $w_k=1$, $k=1,\ldots,N_n$.
\end{itemize}

Observe that in this example the quantity ${\eta}_n={\mbox{dim}}({\mathcal{P}}_n(E))$ is known theoretically, see Remarks \ref{dim for hyperalg} and \ref{dim for alg}. If this is not available, one could use ${\eta}_n={\mbox{rank}}(U_X)$ that in the numerical implementation can be obtained by SVD decomposition of $U_X$ in its less costly version that gives only the singular values (with
a threshold), as used by the rank Matlab/Octave native function. A second alternative could consist in adopting a Rank-Revealing QR factorization algorithm (RRQR). These algorithms, however, are not at hand in standard Matlab and typically require the use of MEX files (see e.g., {\cite{MEX}}).

For the evaluation of the Vandermonde matrix $V_Y$ relatively to the basis $\{\psi_{i}\}_{i=1,\ldots,{\eta}_n}$ on the pointset $Y \subset E$, we first determine the Chebyshev-Vandermonde matrix $U_Y=U(Y)$ and then, in Matlab notation, compute $U_Y(:,\pi)/R$.

From the numerical experiments described in Figure \ref{fig:num_01}, we see that
\begin{enumerate}
\item since $f_1 \in {\mathcal{P}}_{14}$ the interpolants/approximants of degree larger or equal to $14$ provide an approximation close to machine precision;
\item the function $f_2$ is well approximated in view of its regularity;
\item the trigonometric polynomial $f_3$, though is also regular, results more problematic in view of its oscillations;
\item the function $f_4$, has a decreasing error w.r.t. the polynomial degree $n$, but has in general inferior approximation in view of the presence of a square root singularity.
\end{enumerate}

In all these tests the approximation quality is almost the same for interpolation at AFP and DLP, with some better results achieved by least-squares approximation.

\begin{figure}[!h]
        \includegraphics[scale=0.45]{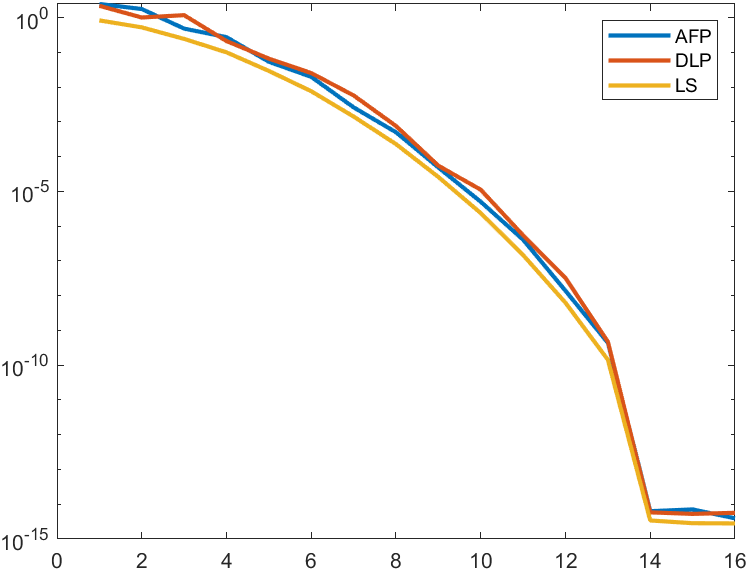}
        \includegraphics[scale=0.45]{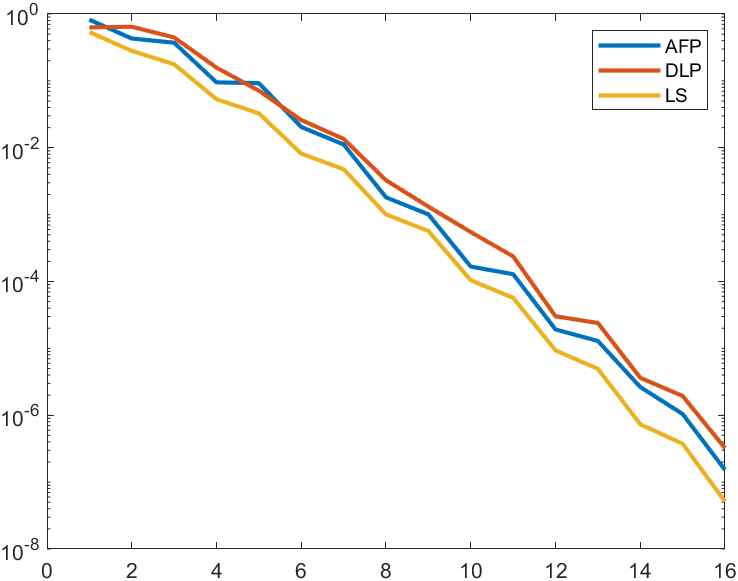} 
        \includegraphics[scale=0.45]{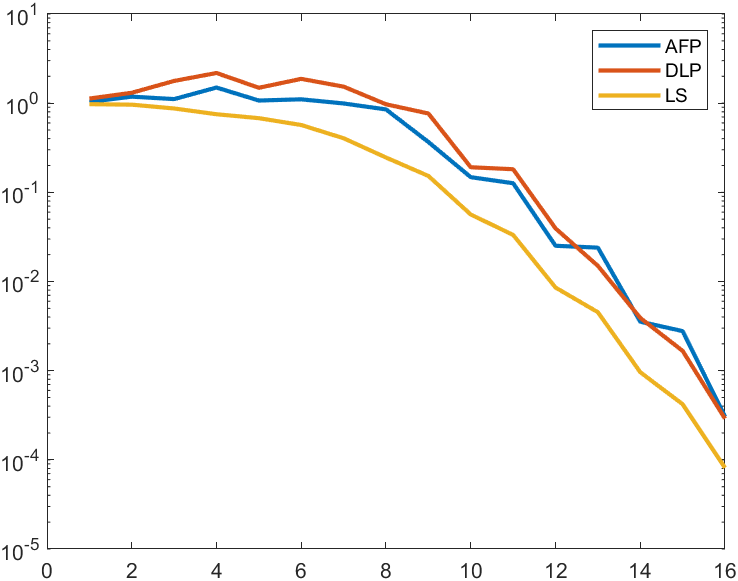}
        \includegraphics[scale=0.45]{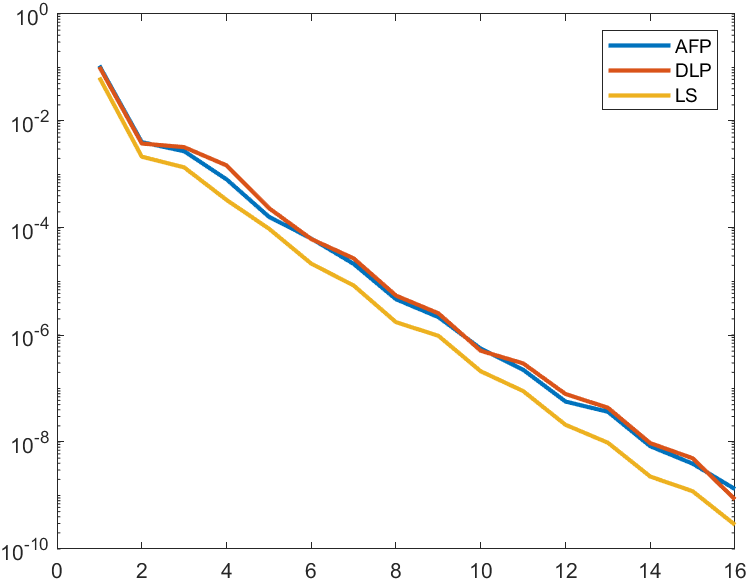}
		\caption{From top to bottom, left to right, we plot the errors in approximating the functions $f_1$, $f_2$, $f_3$, $f_4$ by means of interpolation at Approximate Fekete Points (AFP), Discrete Leja Points (DLP) and least-squares on optimal Admissible meshes ${\mathcal{B}}_n$ for $n=1,2,\ldots,16$, on the set $E$ described in Example {\ref{Ex5}}.}
		\label{fig:num_01}
	\end{figure}

In Figure \ref{fig:num_02}, we plot an approximation of the Lebesgue constants. Our implementation follows that described in \cite[p.435]{BCKSV} for $W=\mbox{diag}({\bf{w}})=I_{N_n}$ and ${\bol}_{30}$ as test points. It is clear how it grows slowly for the least-squares operator, while the AFPs have lower Lebesgue constants than DLPs.

\begin{figure}[!h]
        \includegraphics[scale=0.45]{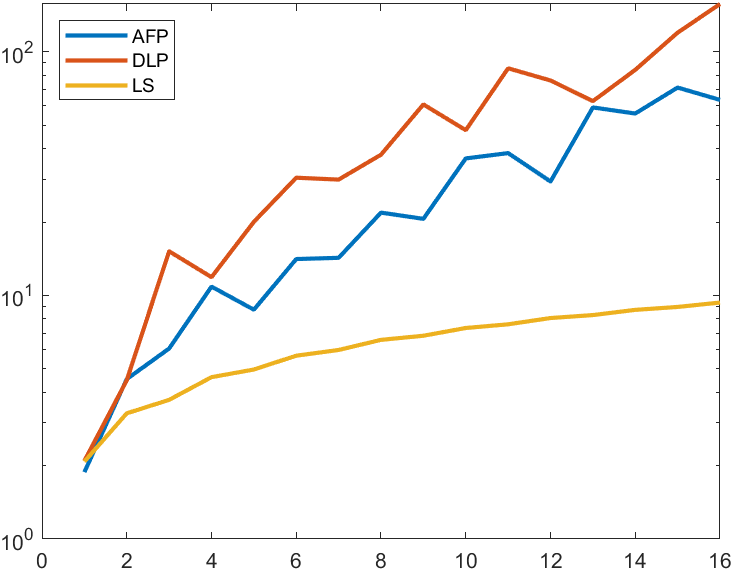}
		\caption{The Lebesgue constant of interpolation at degrees $n=1,2,\ldots,16$ and the least-squares operator {\it{uniform}}-norm for the same $n$, with mesh ${\mathcal{B}}_n$, on the set $E$ described in Example {\ref{Ex5}}.}
		\label{fig:num_02}
	\end{figure}

The analysis of Example \ref{VivWin}, that is the {\it{Viviani's window}}, is very similar to the previous one. We only modified the function $f_4$, considering the sum of the distances from $(0,0,2)$ and $(0,0,-2)$, i.e.
$$
f_4(x,y,z)=\sqrt{x^2+y^2+(z-2)^2}+\sqrt{x^2+y^2+(z+2)^2}.
$$
In Figure {\ref{fig:num_03}}, we display the interpolation/approximation errors and observe that follow the line of those of the previous example. Concerning the Lebesgue constants, from Figure {\ref{fig:num_04}} we see that again the least-squares on ${\mathcal{B}}_{n}$ provide smaller values not far from those of AFPs, while DLPs are larger and more irregular.

\begin{figure}[!h]
        \includegraphics[scale=0.45]{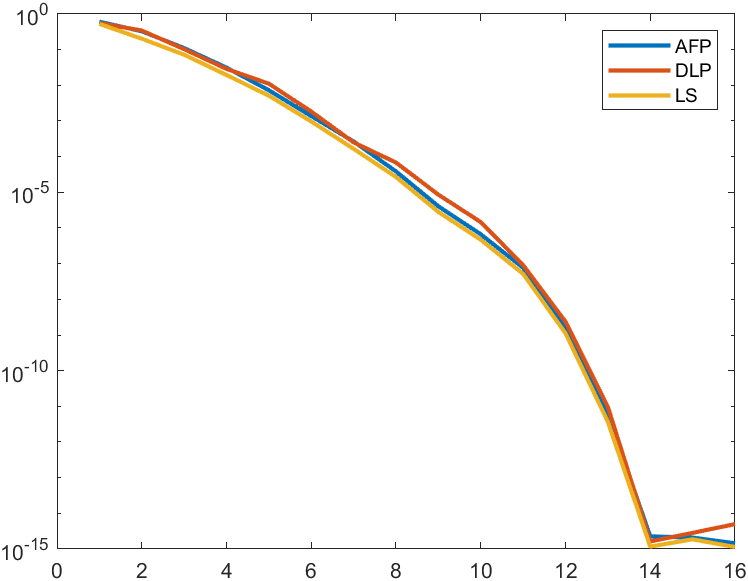}
        \includegraphics[scale=0.45]{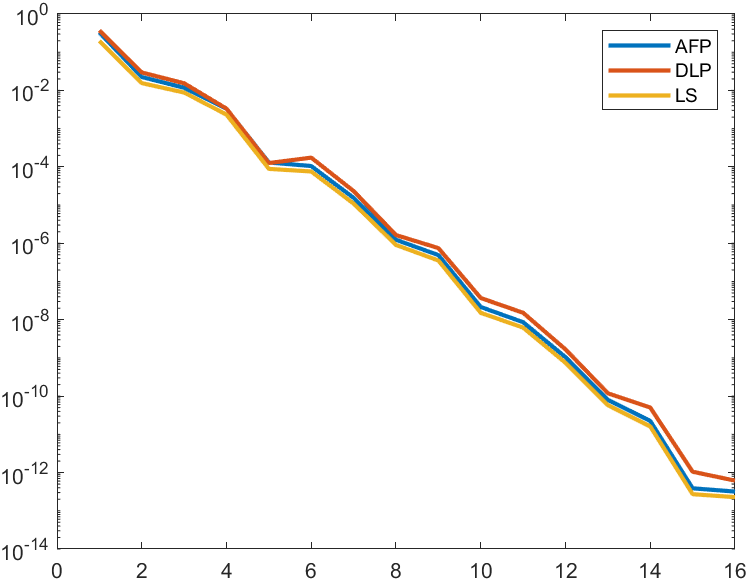}
        \includegraphics[scale=0.45]{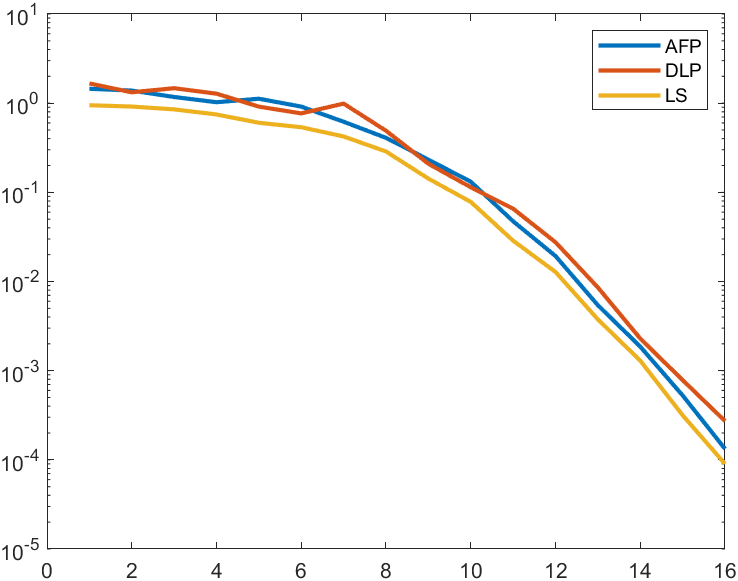}
        \includegraphics[scale=0.45]{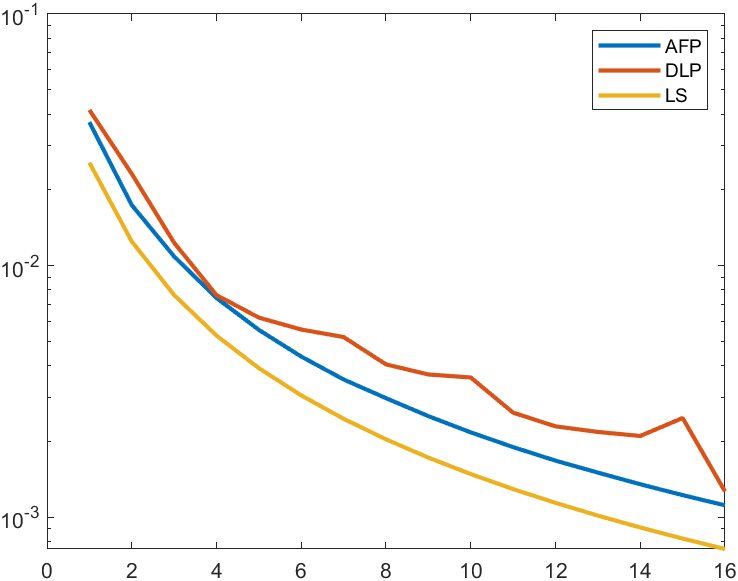}
		\caption{From top to bottom, left to right, we plot the errors in approximating the functions $f_1$, $f_2$, $f_3$, $f_4$ by means of interpolation at Approximate Fekete Points (AFP), Discrete Leja Points (DLP) and least-squares on optimal Admissible meshes ${\mathcal{B}}_n$ for $n=1,2,\ldots,16$, on the set $E$ described in Example {\ref{VivWin}}.}
		\label{fig:num_03}
	\end{figure}

\begin{figure}[!h]
        \includegraphics[scale=0.45]{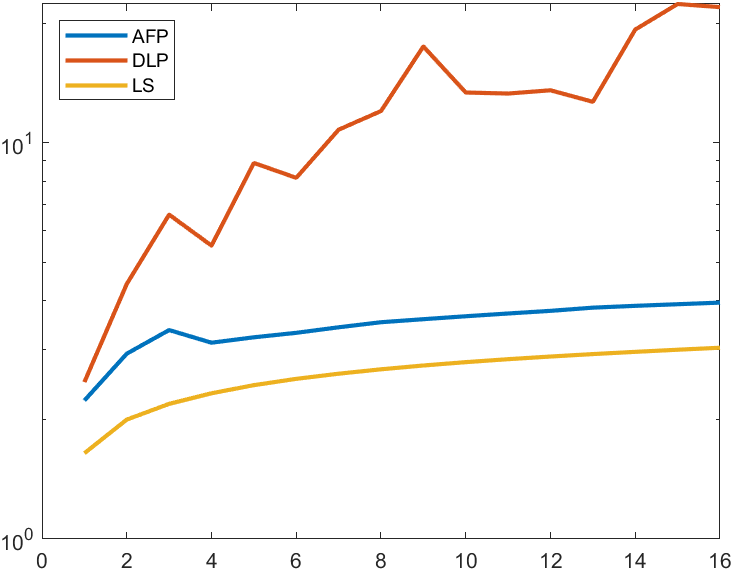}
		\caption{The Lebesgue constant of interpolation at degrees $n=1,2,\ldots,16$ and the least-squares operator {\it{uniform}}-norm for the same $n$, with mesh ${\mathcal{B}}_n$, on the set $E$ described in Example {\ref{VivWin}}.}
		\label{fig:num_04}
	\end{figure}

\vskip 5mm
{\bf Acknowledgements} 
\vskip 1mm

The authors wish to thank Professor Jean-Paul Calvi (University Paul Sabatier in Toulouse, France) for helpful discussion that initiated the research on this paper. Preparation of this article was partially supported by the National Science Centre (Narodowe Centrum Nauki), Poland, grant No. 2017/25/B/ST1/00906, by the program Excellence Initiative – Research University ID UJ at the Jagiellonian University in Kraków, 2024, by the DOR funds of the University of Padova and by the INdAM-GNCS 2024 Project “Kernel and polynomial methods for approximation and integration: theory and application software'' (A. Sommariva). 
Finally, this research has been accomplished within the RITA ``Research ITalian network on Approximation", the UMI Group TAA ``Approximation Theory and Applications", and the SIMAI Activity Group ANA\&A (A. Sommariva).

\end{document}